\documentclass[oneside,11pt]{amsart}

\makeatletter
\@namedef{subjclassname@2020}{
  \textup{2020} Mathematics Subject Classification}
\makeatother
\usepackage[a4paper,width=150mm,top=27mm,bottom=27mm]{geometry}
\usepackage{cases,xcolor}
\usepackage{amsmath,amssymb}
\usepackage[hidelinks]{hyperref}

\newtheorem{theorem}{Theorem}[section]
\newtheorem{lemma}[theorem]{Lemma}

\newtheorem{corollary}[theorem]{Corollary}

\theoremstyle{definition}
\newtheorem{remx}[theorem]{Remark}
\newenvironment{remark}
  {\pushQED{\qed}\remx}
  {\popQED\endremx}

\newcommand{\la}{\langle}
\newcommand{\ra}{\rangle}

\newcommand{\N}{\mathbb{N}}
\newcommand{\R}{\mathbb{R}}
\newcommand{\C}{\mathbb{C}}
\newcommand{\T}{\mathbb{T}}

\newcommand{\vare}{\varepsilon}

\newcommand{\pt}{\partial}
\newcommand{\bg}{\Big}

\newcommand{\Z}{{\mathbb{Z}}}

\newcommand{\px}{P_{N_x}^x}
\newcommand{\py}{P_{N_y}^y}

\numberwithin{equation}{section}

\begin{document}

\address{Yongming Luo
\newline \indent
Faculty of Computational Mathematics and Cybernetics
\newline \indent Shenzhen MSU-BIT University, China}
\email{luo.yongming@smbu.edu.cn}

\title[Critical scattering for the NLS on waveguide manifold]{Critical scattering for the nonlinear Schr\"odinger equation on waveguide manifolds}
\author{Yongming Luo}

\begin{abstract}
We study the small data scattering problem in critical spaces for the nonlinear Schr\"odinger equation (NLS) on waveguide manifolds. Our work is primarily inspired by the recent paper of Kwak and Kwon \cite{KwakKwon} that established the local well-posedness of the periodic NLS with possibly non-algebraic nonlinearity. While we adopt a framework similar to \cite{KwakKwon} for our problem, two main obstacles prevent its direct adaptation to the waveguide setting. First, the classical Strichartz estimates for NLS in critical product spaces, introduced by Hani and Pausader, possess limited endpoints and are thus inapplicable to high-dimensional waveguides. Second, the crucial fractional arguments used in \cite{KwakKwon} rely on a well-known fractional derivative formula due to Strichartz, which admits only a Hilbert space-valued extension and is therefore incompatible with our model setting.

To overcome these difficulties, we develop an anisotropic generalization of the framework in \cite{KwakKwon} using the anisotropic Strichartz estimates established by Tzvetkov and Visciglia, which allow for nearly unlimited endpoints. We also resolve several new challenges arising from the vector-valued and anisotropic nature of the model by employing novel interpolation techniques within Besov spaces. As a further novelty, we provide a new proof of the main result based on classical fixed point arguments, differing from the approximation methods used in \cite{KwakKwon}. Consequently, we settle the small data scattering problem in critical spaces for the NLS with arbitrary mass-supercritical nonlinearity on waveguide manifolds.
\end{abstract}

%\keywords{Critical spaces, Schr\"odinger equation, scattering, waveguide manifold}
%\subjclass[2020]{To be inserted}

\maketitle

%\setcounter{tocdepth}{1}
%\tableofcontents

\section{Introduction and main results}
In this paper, we consider the nonlinear Schr\"odinger equation (NLS)
\begin{align}\label{nls}
(i\pt_t+\Delta_{x,y})u=\pm|u|^{\alpha}u
\end{align}
posed on the waveguide manifold $\R^m\times\T^n$, where $m,n\in\N$ and $\T$ denotes the $2\pi$-torus. Our aim is to establish the following small data scattering result:

\begin{theorem}\label{main thm}
Let $\alpha>\frac{4}{m}$, $d=m+n$, $s_c=\frac{d}{2}-\frac{2}{\alpha}$ and suppose $s_c<1+\alpha$. Then there exists some $\delta>0$ such that for any $u_0\in H^{s_c}(\R^m\times\T^n)$ satisfying $\|u_0\|_{H^{s_c}(\R^m\times\T^n)}\leq \delta$, \eqref{nls} admits a global scattering solution $u\in C(\R;H^{s_c}(\R^m\times\T^n))$ with $u(0)=u_0$. Here, a global scattering solution is referred to as a solution satisfying the following property: there exist $\phi^\pm\in H^s(\R^m\times\T^n)$ such that
\begin{align}\label{scattering_def}
\lim_{t\to\pm\infty}\|u(t)-e^{it\Delta_{x,y}}\phi^\pm\|_{H^s(\R^m\times\T^n)}=0.
\end{align}
\end{theorem}

\subsubsection*{Background and motivation}
In this paper, we primarily study NLS models on the product space $\R^m\times\T^n$, known in the literature as a \textit{waveguide manifold}. This interest stems from a fundamental observation: while global solutions to NLS on Euclidean space with at least mass-critical nonlinearities could scatter in the sense of \eqref{scattering_def}, global solutions on tori of the same dimension generally fail to scatter, see e.g. \cite{global_but_not_scatter} for a concrete counter example in the latter case. This contrast motivates our investigation into whether scattering persists for NLS on waveguide manifolds.

Heuristically, while the partial boundedness of the product space may confine particles within the compact domain, they remain free to propagate infinitely far along the Euclidean directions. This dynamic suggests that possible scattering behavior for global solutions might still be able to take place. Consequently, scattering should only be expected when the Euclidean dimension is sufficiently high. Indeed, scaling analysis as discussed by Hani and Pausader \cite{HaniPausader} indicates that scattering occurs provided the nonlinearity order $\alpha$ satisfies $\alpha \geq \frac{4}{m}$. That is, the nonlinearity must be at least mass-critical with respect to the Euclidean dimension.

The study of NLS on product spaces traces back to the seminal works \cite{TNCommPDE,TzvetkovVisciglia2016} by Tzvetkov and Visciglia, who established well-posedness and scattering results for NLS on $\mathbb{R}^m \times \mathcal{M}^n$ with $\mathcal{M}$ a compact manifold. A key contribution of \cite{TNCommPDE,TzvetkovVisciglia2016} lies in the derivation of suitable Strichartz estimates through delicate spectral analysis along the compact direction. Crucially, these works address the scaling-subcritical regime, where the Sobolev embedding $H^{\frac{n}{2}+} (\mathcal{M}^n)\hookrightarrow L^\infty(\mathcal{M}^n)$ significantly simplifies computations. In contrast, Hani and Pausader \cite{HaniPausader} pioneered the analysis of the scaling-critical case by investigating the quintic NLS in $H^1(\mathbb{R} \times \mathbb{T}^2)$, a model being simultaneously mass- and energy-critical. The critical nature of the model necessitates advanced tools, such as the framework of atomic spaces and mixed discrete-continuous Strichartz estimates. Without being exhaustive, we refer to \cite{RmT1,ZhaoZheng2021,CubicR2T1Scattering,R1T1Scattering,Luo_inter,Luo_energy_crit,luo2021sharp,luo2023sure} for more references in this direction.

\subsubsection*{Challenges and strategies}
In this paper, we address the small data scattering problem in critical spaces of the NLS on waveguide manifolds in the full mass-supercritical regime, a setting being not covered by \cite{TNCommPDE,TzvetkovVisciglia2016} and \cite{HaniPausader}. To see this, we firstly point out that on the one hand, the crucial embedding $H^{\frac{n}{2}+}(\mathcal{M}^n) \hookrightarrow L^\infty(\mathcal{M}^n)$ becomes unavailable, hence eliminating key technical simplification as in \cite{TNCommPDE,TzvetkovVisciglia2016} and necessitating both novel approaches and more sophisticated frameworks. On the other hand, the model under consideration encompasses non-algebraic nonlinearities which are intrinsically incompatible with the multilinear estimates in \cite{HaniPausader} designed exclusively for algebraic cases.

Our approach builds upon the recent work of Kwak and Kwon \cite{KwakKwon} that established local well-posedness for the periodic NLS with mass-supercritical and potentially non-algebraic nonlinearities. We briefly outline their key insight for handling non-algebraic terms: rather than directly utilizing the term-by-term product structure of $|u|^\alpha u$, Kwak and Kwon's strategy is to treat $A = |u|^\alpha$ and $u$ as independent objects and to develop multilinear estimates in terms of both quantities. Crucially, while we are yet unable to obtain cancelation effect in the frequency space as by dealing with the algebraic case, the novel observation of Kwak and Kwon is that certain cancellation can still be gained through considering estimates of $A$ in Besov spaces with negative temporal regularity. This breakthrough enabled them to resolve the local well-posedness problem for periodic NLS in the full mass-supercritical regime.

While we adopt a similar framework for our problem, several new difficulties emerge in the waveguide setting. As first attempts, comparing the approaches in \cite{HaniPausader} and \cite{KwakKwon} might suggest that replacing the periodic Strichartz estimates from \cite{KwakKwon} with their waveguide counterparts from \cite{HaniPausader} would suffice. Such direct substitution will nevertheless fail. To see this, we recall the Strichartz estimate established in \cite{HaniPausader} (see also \cite[Thm. 1]{Barron}):
$$ \|e^{it\Delta_{x,y}} f\|_{\ell^q_\gamma L_{t,x,y}^p([\gamma-\pi,\gamma+\pi]\times\R^m\times\T^n)}
\lesssim \|f\|_{H^s(\R^m\times\T^n)}$$
with exponents $p>2+\frac{4}{d}$, $q=\frac{4p}{m(p-2)}$ satisfying $q>2$ and $s=\frac{d}{2}-\frac{d+2}{p}$. At this point, we underline that the constraint $q>2$ requires  $p<2+\frac{4}{m-2}$ when $m\geq 3$. This becomes problematic for $\alpha\gg 1$: the Lebesgue exponent $\frac{\alpha(d+2)}{2}$ appearing in the nonlinear estimation becomes unbounded as $\alpha\to\infty$, which eventually exceeds the number $2+\frac{4}{m-2}$ and invalidates the applicability of the Strichartz estimates.

To resolve this technical obstacle, our pivotal observation is that the Strichartz estimates due to Tzvetkov and Visciglia \cite{TNCommPDE,TzvetkovVisciglia2016}
\begin{align}\label{example strichartz}
\|e^{it\Delta_{x,y}}f\|_{L_t^p L_x^q L_y^2(\mathbb{R}\times\mathbb{R}^m\times\mathbb{T}^n)} \lesssim \|f\|_{L^2(\mathbb{R}^m\times\mathbb{T}^n)},
\end{align}
where $(p,q)$ is an admissible Strichartz pair, admit almost unlimited endpoints in the sense that the number $p$ can be taken arbitrarily large. Leveraging this flexibility, our plan is to employ the Strichartz estimate \eqref{example strichartz} to prove our main small data scattering result. The anisotropy of \eqref{example strichartz} nevertheless creates additional incompatibilities with Kwak and Kwon's framework. Specifically, the high-low frequency analysis given in \cite{KwakKwon} relies on Visan's fractional chain rule for H\"older continuous functions \cite{defocusing5dandhigher}, which depends fundamentally on Strichartz's difference characterization of fractional derivatives \cite{Strichartz_derivative}. In our anisotropic setting, derivatives taken along Euclidean directions require a vector-valued generalization of Strichartz's result. As pointed out in \cite{WalkerThesis}, such generalizations are only feasible for Hilbert target spaces, which being a condition violated in our setting since the underlying periodic target space follows an $L^\infty$-scaling. To overcome this limitation, we develop new fractional chain and product rules within Besov spaces that serve as more accommodating function spaces in vector-valued settings.

Finally, we shall conclude by proving Theorem \ref{main thm} using classical fixed-point arguments which differ from the approximation methods applied in \cite{KwakKwon}. This approach leverages the completeness of the underlying metric space under the weaker norm $Y^0$, which admits stronger multilinear estimates than those available in the smaller space $Y^{s_c}$. We refer to Lemma \ref{lem7.1} for further details.

\subsubsection*{Some final remarks}

We end this introductory section by giving several concluding remarks.

\begin{remark}\label{rem1.2}
We emphasize that in Theorem \ref{main thm}, we are unable to establish a small-data scattering result in the endpoint case $\alpha=\frac{4}{m}$. The main reason lies in the fact that when applying Kwak and Kwon's framework, one must consider functions in certain Besov spaces with positive temporal regularity, hence not fitting with the mass-critical scaling. Notably, unlike the periodic case, where mass-critical Strichartz estimates are known to possess necessary derivative losses \cite{Bourgain1,KwakPi}, making endpoint well-posedness unlikely to hold (at least in the classical sense), the anisotropic Strichartz estimates on product spaces do actually hold at the mass-critical endpoint. In fact, for waveguide manifolds, endpoint small data scattering has been established for algebraic nonlinearities (see e.g. \cite{HaniPausader}), where the cancellation effects of the Schr\"odinger operator can be fully employed. From this perspective, we conjecture that Theorem \ref{main thm} should remain valid in the mass-critical endpoint case.
\end{remark}

\begin{remark}
The constraint $s_c<1+\alpha$ is natural since the regularity order of the nonlinearity $|u|^\alpha u$ is smaller than $1+\alpha$.
\end{remark}

\begin{remark}
Since the proof of Theorem \ref{main thm} proceeds via fixed-point arguments within a metric space $\mathcal{K}$, the global scattering solution constructed herein is in fact unique within $\mathcal{K}$. The question of whether this solution is also unique in the class $C_t H_{x,y}^{s_c}$ is known in the literature as the \textit{unconditional uniqueness} problem. For NLS posed on Euclidean spaces, such unconditional uniqueness problems can be easily solved by a straightforward application of the Strichartz estimates, see e.g. \cite{defocusing3d}. On the contrary, for NLS on compact manifolds, the absence of well-established Strichartz estimates renders this problem far more challenging. Recent advances in \cite{uncond1,uncond2,uncond3} have resolved the unconditional uniqueness problem for the cubic and quintic periodic NLS with at least energy-critical nonlinearities. Nevertheless, the framework developed in these works is unlikely to extend to NLS models with non-algebraic nonlinearities, as it relies fundamentally on linear structures inherent to quantum many-body systems that are invalidated for non-algebraic interactions. This limitation thereby highlights intriguing open problems for future investigation, which we do not elaborate on further here.
\end{remark}

\begin{remark}
It is worth mentioning that Kwak recently established in \cite{kwak_large_data} the large data well-posedness result for the periodic energy-critical NLS in arbitrary dimension $d\geq 3$, building on the prior work \cite{KwakKwon}. Crucially, the main contribution of \cite{kwak_large_data} is demonstrating that concentration compactness arguments remain applicable even when classical stability theories fail. Inspired by this result, we also expect a large data scattering result for \eqref{nls} in the energy-critical case ($s_c=1$) will continue to hold. While we believe this should be achievable for $n=1$ (where the model is mass-supercritical w.r.t. the Euclidean dimension) by combining arguments from \cite{kwak_large_data} with the semivirial-vanishing geometry theory developed in the author's series of works \cite{Luo_inter,Luo_energy_crit,luo2021sharp}, we encounter a major new challenge in the mass-critical case $n=2$. Here, not only is the small data scattering result currently unknown (as pointed out by Remark \ref{rem1.2}), but we also lack a clear understanding of the corresponding fractional resonant system -- a key tool for studying large data scattering for mass-critical NLS on waveguide manifolds (see e.g. \cite{HaniPausader}) -- which is not even well-defined.
%As they are out of scope of the present paper, we shall leave these interesting problems for future research.
\end{remark}

\section{Notation}\label{sec 2.1}
We use the notation $A\lesssim B$ whenever there exists some positive constant $C$ such that $A\leq CB$. Similarly we define $A\gtrsim B$ and use $A\sim B$ when $A\lesssim B\lesssim A$. For a $k$-dimensional vector $\xi\in \C^k$, the Japanese bracket is defined by $\la\xi\ra:=(1+|\xi|^2)^{\frac12}$.

As usual, we use the symbol $\mathcal{F}(f)$ or $\widehat{f}$ to denote the Fourier transformation of a function $f$. Since we will be dealing with functions defined on a product space $\R^m\times\T^n$, we shall use $\mathcal{F}_x, \mathcal{F}_y$ (with $x\in\R^m$ and $y\in\T^n$) as well as $\mathcal{F}$ to denote the Fourier transformation along partial or total directions respectively.

Finally, the fractional derivative operator $D_x^s$ with $s\in \R$ is defined via its symbol $\mathcal{F}_x(D_x^s)(\xi)=|\xi|^s$ for $\xi\in\R^m$. The periodic fractional derivative operator $D_y^s$ is defined analogously.

\subsubsection*{Function spaces}
Given a Banach space $E$, numbers $k\in\N$, $m\in\N_0$, $p\in[1,\infty)$ and a set $\Omega\subset\R^k$, we define the following norms for a function $f:\Omega\to E$:
\begin{align*}
\|f\|_{L_z^p E}:=\bg(\int_{\Omega}\|f(z)\|_E^p\,dz\bg)^{\frac{1}{p}},\quad
\|f\|_{W_z^{m,p} E}:=\sum_{|\alpha|\leq m}\|\pt_z^\alpha f\|_{L^p E}.
\end{align*}
The norms in the case $p=\infty$ are defined by convention. We will mainly consider the cases $z\in\{t,x,y\}$ which stand for $\Omega\in\{\R,\R^m,\T^n\}$. For $s\in\R$ and $p\in(1,\infty)$, we shall also use the Sobolev norm $\|f\|_{H^{s,p}_y}:=\|\mathcal{F}_y^{-1}(\la\xi\ra^s\mathcal{F}_y(f)(\xi))\|_{L^p}$ via the Bessel potential.

For the sake of keeping the paper as short and concise as possible, the temporal or spatial integration domains will not be written explicitly in most cases, e.g. the space $L_t^{p_1} L_x^{p_2} L_y^{p_3}(\R\times\R^m\times\T^n)$ will be abbreviated as $L_t^{p_1} L_x^{p_2} L_y^{p_3}$ or $L_t^{p_1} L_x^{p_2} L_y^{p_3}(\R)$, where the one-dimensional Euclidean space $\R$ in the latter is referred to as the temporal integration domain.

Finally, for $s\in\R$, $p\in[1,\infty]$ we define the discrete weighted Lebesgue spaces
\begin{align*}
\ell_N^{s,p}:=\{f:2^\N \to\C:\|f\|_{\ell^{s,p}_N}:=\bg(\sum_{N\in 2^\N} N^{sp}|f_N|^p\bg)^{\frac{1}{p}}<\infty\}.
\end{align*}
When $s=0$, we also denote the space $\ell_N^{s,p}$ by $\ell^p_N$. The space $\ell^p_j$ with $j\in A\subset\N$ is defined in a similar fashion.

\subsubsection*{Littlewood-Paley projectors}
We fix $\eta_1\in C^\infty_c(\R;[0,1])$ to be a one-dimensional bump function such that $\eta(t)\equiv 1$ for $|t|\leq 1$ and $\eta(t)\equiv 0$ for $|t|\geq 0$. For $k\in\N$ define $\eta_k (t_1,...t_k):=\prod_{i=1}^k \eta_1 (t_i)$. Then the Littlewood-Paley (LP) projectors $P_{\leq N}$ and $P_N$ on $\Omega^k$ with $\Omega\in\{\R,\T\}$ for a dyadic number $N\in 2^\N$ with $N\geq 2$ are defined via their symbols
$$ \mathcal{F}(P_{\leq N})(\xi)=\eta_d(\xi/N),\quad P_N:= P_{\leq N}-P_{\leq N/2},$$
where for $N=1$ we shall simply define $P_{1}:=P_{\leq 1}:=\mathcal{F}^{-1}(\eta_d)$. In our paper we shall use the LP-projectors in an anisotropic way, i.e. we make use of the LP-projectors along different directions. We use different superscripts to indicate the dependence of the LP-projectors on their arguments, e.g. $P_N^t$, $P_{\leq N}^x$ etc. Moreover, the symbol $P$ without superscripts is referred to as the LP-projector defined on the whole spatial domain $\R^m\times\T^n$, e.g. $P_{\leq N}:=P^x_{\leq N}P^y_{\leq N}$. By telescoping arguments one easily verifies that
\begin{align}\label{2.1}
P_N=P_N^x P_{\leq N}^y+P_{\leq N/2}^x P_N^y.
\end{align}
As we shall see, the identity \eqref{2.1} will play a crucial role in the anisotropic analysis appearing in the upcoming proofs.

\subsubsection*{Fixed and small numbers}
Throughout the paper, we define $p_0$ as the fixed number $p_0:=2+\frac{4}{m}$. Moreover, the following small numbers will be used in increasing order: $0<\sigma\ll \sigma_1\ll 1$.

\subsubsection*{Admissible Strichartz pair}
For $s\in(0,\infty)$, a pair $(q,r)$ is said to be an $s$-admissible Strichartz pair (w.r.t. the Euclidean dimension $m$) if $q,r\in[2,\infty]$, $\frac{2}{q}+\frac{m}{r}=\frac{m}{2}-s$ and $(q,r,m)\neq(2,\infty,2)$. When $\kappa=0$ we shall simply call the pair an admissible Strichartz pair.

\subsubsection*{Atomic spaces and their properties}
Next, for $s\in\R$ we introduce the spaces $X^s$ and $Y^s$.
%Let $u:\R_t\times (\R_x^3\times\T_y)\to\C$ be a measurable function.
Denote by $C=(-1/2,1/2]^d$ the unit cube in $\R^d$. For $z\in\R^d$ the translated cube $C_z$ is defined by $C_z:=C+z$. Moreover, we define the projector $P_{C_z}$ by
$$\mathcal{F}(P_{C_z}u):=\chi_{C_z}\mathcal{F}(u),$$
where $\chi_{C_z}$ is the characteristic function of $C_z$. For $s\in\R$ we then define the spaces $X_0^s(\R)$ and $Y^s(\R)$ through the norms
\begin{align*}
\|u\|_{X^s_0(\R)}^2:=\sum_{z\in\Z^d}\la z\ra^{2s}\|P_{C_z} u\|_{U_{\Delta_{x,y}}^2(\R;L^2_{x,y})},\\
\|u\|_{Y^s(\R)}^2:=\sum_{z\in\Z^d}\la z\ra^{2s}\|P_{C_z} u\|_{V_{\Delta_{x,y}}^2(\R;L^2_{x,y})},
\end{align*}
where $U_{\Delta_{x,y}}^2$ and $V_{\Delta_{x,y}}^2$ are the standard atom spaces taking values in $L_{x,y}^2$ (see for instance \cite{HadacHerrKoch2009} for their precise definitions). For any subinterval $I\subset\R$, the space $X^s(I)$ is defined through the norm
\begin{align*}
\|u\|_{X^s(I)}:=\inf\{\|v\|_{X^s_0(\R)}:v\in X_0^s(\R),\,v|_I=u|_I\}.
\end{align*}
The space $Y^s(I)$ is similarly defined. For an interval $I=(a,b)$, the space $N^s(I)$ is defined through the norm
\begin{align*}
\|u\|_{N^s(I)}:=\bg\|\int_a^t e^{i(t-\sigma)\Delta_{x,y}}u(\sigma)\,d\sigma\bg\|_{X^s(I)}.
\end{align*}
We record the following useful properties of the previously defined function spaces.

\begin{lemma}[Embeddings between function spaces, \cite{HadacHerrKoch2009,HerrTataruTz1}]\label{embedding lem}
For any $s\in\R$ and $p\in(2,\infty)$ we have
\begin{align*}
U^2_{\Delta_{x,y}}(I;H^s_{x,y})\hookrightarrow X^s(I)\hookrightarrow Y^s(I)\hookrightarrow V_{\Delta_{x,y}}^2(I;H_{x,y}^s)
\hookrightarrow U^p_{\Delta_{x,y}}(I;H_{x,y}^s)\hookrightarrow L_t^\infty(I;H_{x,y}^s).
\end{align*}
\end{lemma}

\begin{lemma}[Duality of $N^s$ and $Y^{-s}$, \cite{HerrTataruTz1}]\label{dual lem}
For $u\in L_t^1H_{x,y}^1(I)$ we have
\begin{align*}
\|u\|_{N^s(I)}\lesssim \sup_{\|v\|_{Y^{-s}(I)}=1}\bg|\int_{I\times(\R^m\times\T^n)}u(t,x,y)\bar{v}(t,x,y)\,dxdydt\bg|.
\end{align*}
\end{lemma}

\section{Some preliminary tools}

\subsection{Vector-valued Besov spaces}
Given a Banach space $E$, the numbers $k\in\N$, $s\in\R$, $p,q\in[1,\infty]$ and the argument $z\in\{\R^k,\T^k\}$, the vector-valued Banach space $B_{z,p,q}^s$ is defined via the norm
$$ \|u\|_{B_{z,p,q}^s E}:=\bg(\sum_{N\in 2^\N}N^{qs}\|P_N^z(u)\|_{L_z^p E}^q\bg)^{\frac{1}{q}}+\|P_1^z u\|_{L_z^p E}.$$
For our purpose, we will alternatively make use of the following well-known characterization of the vector-valued Besov spaces.

\begin{lemma}[Difference characterizations of Besov spaces, \cite{amannbesov,Triebel_Functionspaces}]\label{lem_triebel}
The following statements hold true:
\begin{itemize}
\item[(i)]For $s\in(0,1)$, $p\in(1,\infty)$, $k\in\N$ and $\mathcal{M}\in\{\R,\T\}$ we have
\[\|u\|^p_{B_{p,p}^s(\mathcal{M}^k)}\sim \|u\|_{L^p(\mathcal{M}^k)}^p
+\int_{\mathcal{M}^k\times \mathcal{M}^k}\bg(\frac{|u(z_1)-u(z_2)|}{|z_1-z_2|^s}\bg)^p\,\frac{dz_1 dz_2}{|z_1-z_2|^{k}}.\]

\item[(ii)] Let $E$ be a Banach space. Under the same conditions for $s,p,k$ as in (i) it holds
\[\|u\|^p_{B_{p,p}^s(\R^k;E)}\sim \|u\|_{L^p(\R^k;E)}^p
+\int_{\R^k\times \R^k}\bg(\frac{\|u(z_1)-u(z_2)\|_{E}}{|z_1-z_2|^s}\bg)^p\,\frac{dz_1 dz_2}{|z_1-z_2|^{k}}.\]

\item[(iii)]Let $s\in (0,\infty)$, $p\in[1,\infty)$, $q\in [1,\infty]$, $k\in \N$, $E$ a Banach space, and let $s=[s]^{-}+\{s\}^{+}$ with $[s]^{-}\in\Z$ and $\{s\}^{+}\in[0,1)$. For $h\in \R^k$ and $u:\R^k\to E$ define
\[\Delta_h^1 u(x):=u(x+h)-u(x),\quad \Delta_h^2 u(x):=(\Delta_h^1(\Delta_h^{1} u))(x).\]
Moreover, define
\begin{align*}
[u]_{s,p,q,E}:=
\left\{
\begin{array}{ll}
\sum_{|\alpha|=[s]^{-}}\bg(\int_{\R^k} |h|^{-\{s\}^+q}\|\Delta_h^2 D^\alpha f\|_{L^{p}(\R^k;E)}^q\, \frac{dh}{|h|^n}\bg)^{\frac{1}{q}},&q<\infty,\\
\sum_{|\alpha|=[s]^{-}}\sup_{h\in \R^k\setminus\{0\}}|h|^{-\{s\}^{+}}\|\Delta_h^2 D^\alpha f\|_{L^{p}(\R^k;E)},&q=\infty.
\end{array}
\right.
\end{align*}
Then
\[\|u\|_{B_{p,q}^s (\R^k;E)}\sim \|u\|_{W^{[s]^{-},p}(\R^k;E)}+[u]_{s,p,q,E}.\]
\end{itemize}
\end{lemma}

Lemma \ref{lem_triebel} implies immediately the following corollaries.
\begin{corollary}\label{cor 3.2}
Let $\alpha\in(0,1]$ and $F,G:\C\to\C$ be $\alpha$-H\"older continuous functions with $F(0)=G(0)=0$. Then for $s\in(0,1)$, $k\in\N$ and $p_1,p_2\in(1,\infty)$ it holds $\|F(u)\|_{B_{y,p_1/\alpha,p_1/\alpha}^{\alpha s}}\lesssim
\|u\|^\alpha_{B_{y,p_1,p_1}^s }$ and
$\|G(u)\|_{B_{x,p_1/\alpha,p_1/\alpha}^{\alpha s}L_y^{p_2/\alpha}}\lesssim
\|u\|^\alpha_{B_{x,p_1,p_1}^s L_y^{p_2}}$.
\end{corollary}

\begin{proof}
The first estimate is proved in \cite[Lem. 2.7]{KwakKwon}. We follow the same lines of \cite[Lem. 2.7]{KwakKwon} to prove the second estimate. By Lemma \ref{lem_triebel} (ii) we have
\begin{align*}
\|G(u)\|^{p_1/\alpha}_{B_{x,p_1/\alpha,p_1/\alpha}^{\alpha s}L_y^{p_2/\alpha}}&\sim
\|G(u)\|^{p_1/\alpha}_{L_x^{p_1/\alpha}L_y^{p_2/\alpha}}\\
&\quad+
\int_{\R^k\times \R^k}\bg(\frac{\|(G(u))(z_1)-(G(u))(z_2)\|_{L_y^{p_2/\alpha}}}{|z_1-z_2|^s}\bg)^{p_1/\alpha}\,\frac{dz_1 dz_2}{|z_1-z_2|^{k}}\\
&=:I+II.
\end{align*}
Using the H\"older-continuity of $G$ and the condition $G(0)=0$ we infer that
$$ I=\|G(u)-G(0)\|^{p_1/\alpha}_{L_x^{p_1/\alpha}L_y^{p_2/\alpha}}\leq \|u\|_{L_x^{p_1}L_y^{p_2}}^{p_1}$$
and
\begin{align*}
II&\leq
\int_{\R^k\times \R^k}\bg(\frac{\||u(z_1)-u(z_2)|^\alpha\|_{L_y^{p_2/\alpha}}}{|z_1-z_2|^s}\bg)^{p_1/\alpha}\,\frac{dz_1 dz_2}{|z_1-z_2|^{k}}\\
&\leq
\int_{\R^k\times \R^k}\bg(\frac{\|u(z_1)-u(z_2)\|_{L_y^{p_2}}}{|z_1-z_2|^{s/\alpha}}\bg)^{p_1}\,\frac{dz_1 dz_2}{|z_1-z_2|^{k}},
\end{align*}
from which the claim follows.
\end{proof}

\begin{corollary}\label{cor 3.3}
Let $s\in(0,\infty)$, $p\in[1,\infty)$ and $q\in[1,\infty]$. Let also $E,F$ be Banach spaces with $E\hookrightarrow F$. Then $B_{p,q}^s (\R^k;E)\hookrightarrow B_{p,q}^s (\R^k;F)$.
\end{corollary}

\begin{proof}
This follows immediately from Lemma \ref{lem_triebel} (iii).
\end{proof}

The following lemma gives the fundamental properties of the vector-valued Besov spaces.

\begin{lemma}[Embedding, duality and interpolation for Besov spaces, \cite{amannbesov,AmannEmbedding,NakamuraWada}]\label{lemallprop}
Let $E$, $E_1$, $E_2$ be Banach spaces and $B_{p,q}^s$ denote Besov-spaces defined on $\R^k$ with values in $E$. For given numbers $p_i,q_i,s_i$, $i=0,1$, define
$p_\theta^{-1}=(1-\theta)p_0^{-1}+\theta p_1^{-1}$, $q_\theta^{-1}=(1-\theta)q_0^{-1}+\theta q_1^{-1}$ and $s_\theta=(1-\theta)s_0+\theta s_1$. Then the following statements hold true:
\begin{itemize}
\item[(i)] For $m\in\Z$ and $p\in[1,\infty)$ we have $B_{p,1}^m E\hookrightarrow W^{m,p} E\hookrightarrow B_{p,\infty}^m E$.

\item[(ii)] For $1\leq p_1<p_2<\infty$ and $q\in[1,\infty]$ we have $B_{p_1,q}^{k(\frac{1}{p_1}-\frac{1}{p_2})} E\hookrightarrow L^{p_2,q} E$.

\item[(iii)] For $1\leq p_1<p_2<\infty$ and $M\in 2^{\N_0}$ we have $\|P_M f\|_{L^{p_2} E}\lesssim M^{k(\frac{1}{p_1}-\frac{1}{p_2})}\|P_M f\|_{L^{p_1} E}$.
\item[(iv)] For $1\leq p_1<p_2<\infty$, $q\in[1,\infty]$ and $s\in\R$ we have $B_{p_1,q}^{s+k(\frac{1}{p_1}-\frac{1}{p_2})} E\hookrightarrow B_{p_2,q}^{s} E$.

\item[(v)] For $p,q\in[1,\infty)$, $s\in\R$, and either $E$ being reflexive or $E'$ being separable\footnote{A Banach space satisfying such property is referred to as a Banach space satisfying the \textit{Radon-Nikodym property} in literature.}, we have $(B^s_{p,q}E)'=B^{-s}_{p',q'}E'$.

\item[(vi)] For $p\in[1,\infty)$, $q_0,q_1,\eta\in[1,\infty]$, $\theta\in(0,1)$ and $s_0,s_1\in\R$ with $s_0\neq s_1$ we have the real interpolation $(B_{p,q_0}^{s_0} E, B_{p,q_1}^{s_1} E)_{\theta,\eta}=B_{p,\eta}^{s_\theta}E$.
\item[(vii)] For $p_0,p_1\in(1,\infty)$, $q_0,q_1\in[1,\infty]$, $s_0,s_1\in\R$ and an interpolation couple $(E_0,E_1)$ we have the complex interpolation
$[B_{p,q_0}^{s_0} E_0, B_{p,q_1}^{s_1} E_1]_{\theta}=B_{p_\theta,q_\theta}^{s_\theta}[E_0,E_1]_\theta$.
\end{itemize}
\end{lemma}

\subsection{Some tools from fractional calculus}\label{sec 3.2}
We firstly record a useful frequency-localized Besov-H\"older type inequality introduced in \cite[Prop. 2.3]{KwakKwon}.

\begin{lemma}[\cite{KwakKwon}]\label{kwak_lem_lhh}
Let $E_i$, $i\in\{1,2,3\}$, be Banach spaces defined on a measure space $\mathcal{M}$ and assume the inequality
\begin{align}\label{condEi}
\bg|\int_{\mathcal{M}} f_1f_2f_3\,dm\bg|\lesssim \|f_1\|_{E_1}\|f_2\|_{E_2} \|f_3\|_{E_3}
\end{align}
holds. Let also $s_j\in\R$, $p_j\in(1,\infty)$, $q_j\in[1,\infty]$, $j\in\{1,2,3\}$, satisfy
\begin{align*}
&s_1+s_2+s_3>0,\quad s_2+s_3>0,\quad \frac{k}{p_1}>s_1,\\
&\frac{1}{p_1}+\frac{1}{p_2}+\frac{1}{p_3}=1+\frac{s_1+s_2+s_3}{k},\quad \frac{1}{q_1}+\frac{1}{q_2}+\frac{1}{q_3}=1
\end{align*}
with $k\in\N$. Then
\begin{align*}
\sum_{\substack{L,M,N\in 2^{\N_0}\\
L\lesssim M\lesssim N}}
\bg|\int_{\mathcal{M}\times\mathcal{D}^k } P_L f P_M g P_N h\,dm dz\bg|\lesssim
\|f\|_{B_{z,p_1,q_1}^{s_1}E_1}
\|g\|_{B_{z,p_2,q_2}^{s_2}E_2}
\|h\|_{B_{z,p_3,q_3}^{s_3}E_3},
\end{align*}
where $\mathcal{D}\in\{\R,\T\}$.
\end{lemma}

By duality, Lemma \ref{kwak_lem_lhh} yields immediately the following corollary.
\begin{corollary}\label{cor 3.6}
Let Banach spaces $E_i$, $i\in\{1,2,3\}$ satisfy \eqref{condEi}. Let also $s_j\in\R$, $p_j\in(1,\infty)$, $q_j\in[1,\infty]$, $j\in\{1,2,3\}$, satisfy
\begin{align*}
&s_1+s_2>s_3,\quad \min\{s_1,s_2\}>s_3,\\
&\frac{k}{p_1}>s_1,\quad \frac{k}{p_2}>s_2,\quad \frac{k}{p_3}<k+s_3\\
&\frac{1}{p_3}-\frac{s_3}{k}=\frac{1}{p_1}-\frac{s_1}{k}+\frac{1}{p_2}-\frac{s_2}{k},\quad
\frac{1}{q_3}=\frac{1}{q_1}+\frac{1}{q_2}
\end{align*}
with $k\in\N$. Then
\begin{align*}
\|uv\|_{B^{s_3}_{z,p_3,q_3}E_3'}\lesssim \|u\|_{B_{z,p_1,q_1}^{s_1}E_1}\|v\|_{B_{z,p_2,q_2}^{s_2}E_2}.
\end{align*}
\end{corollary}

The following lemma is a vector-valued generalization of \cite[Lem. 2.8]{KwakKwon}. The modification is straightforward, we omit the details.

\begin{lemma}[\cite{KwakKwon}]\label{lem 3.6}
Let $s_t,s_x>0$ be exponents satisfying $2s_t+s_x<1$. Fix $p_{tx},p_y\in(1,\infty)$, $\alpha\in(0,1)$ and $F\in C^{0,\alpha}(\C)$ with $F(0)=0$. Then
\begin{align}
\|F(u)\|_{B^{s_t\alpha}_{t,p_{tx}/\alpha,p_{tx}/\alpha}B^{s_x\alpha}_{x,p_{tx}/\alpha,p_{tx}/\alpha}
L_y^{p_y/\alpha}}\lesssim
\|u\|^{\alpha}_{L_t^{p_{tx}} B_{x,p_{tx},p_{tx}}^{2s_t+s_x} L_y^{p_y}\cap
B_{t,p_{tx},p_{tx}}^{s_t+s_x/2} L_x^{p_{tx}} L_y^{p_y}}.
\end{align}
\end{lemma}

While with the help of Lemma \ref{lem 3.6} we are able to deal with derivatives taken along the $x$-direction, the proof of Lemma \ref{lem 3.6} can not be employed to handle the $y$-derivatives since the exponents $p_{tx}$ and $p_y$ are in general different from each other. To that end, we give the following key lemma for estimating derivatives involving periodic directions.

\begin{lemma}\label{lem 3.7}
Given $s_t,s_y>0$ let $p_{t}, p_y,p\in(1,\infty)$ satisfy $\min\{\frac{2s_t p_{t}}{2s_t+s_y},
\frac{s_y p_{y}}{2s_t+s_y}\}\geq 1$. Then
\begin{align*}
\|u\|_{B^{s_t}_{t, p_{t},p_{t}}L_x^{p} B^{s_y}_{t, p_{y},p_{y}}}
\lesssim
\|u\|_{L_{t}^{p_{t}}L_x^p B_{y,p_y,\frac{s_yp_y}{2s_t+s_y}}^{2s_t+s_y}\cap
B_{t,p_{t},\frac{2s_tp_{t}}{2s_t+s_y}}^{s_t+s_y/2} L_x^{p} L_y^{p_y}}.
\end{align*}
\end{lemma}

\begin{proof}
Using Lemma \ref{lemallprop} (vii) we know that
\begin{align*}
B_{t,p_{t},\infty}^0 L_x^{p} B_{y,p_y,\frac{s_yp_y}{2s_t+s_y}}^{2s_t+s_y}
\cap B_{t,p_{t},\frac{2s_tp_{t}}{2s_t+s_y}}^{s_t+s_y/2} L_x^{p} B_{y,p_y,\infty}^{0}
\hookrightarrow
B^{s_t}_{t, p_{t},p_{t}}L_x^{p} B^{s_y}_{t, p_{y},p_{y}}.
\end{align*}
The desired claim follows from combining the embedding $L^p\hookrightarrow B_{p,\infty}^0$ given by Lemma \ref{lemallprop} (i) and Corollary \ref{cor 3.3}.
\end{proof}

Next, we introduce the crucial lemmas for handling the low-high interactions appearing in the nonlinear estimation given in Section \ref{sec 5+}. The first estimate addresses the terms involving derivatives in periodic directions.

\begin{lemma}[Fractional chain rule for periodic functions, \cite{KwakKwon}]\label{lem3.9}
Let $\alpha\in(1,\infty)$, $s\in[0,\alpha)$ and $m\in\Z$. Let also $p,p_1,p_2\in(1,\infty)$ satisfy $p^{-1}=(\alpha-1)p_1^{-1}+p_2^{-1}$. Then for $u:\T^n\to \C$ we have
\begin{align*}
\||u|^{\alpha-k} u^k\|_{H^{s,p}_y}\lesssim \|u\|_{L_y^{p_1}}^{\alpha-1}\|u\|_{H_y^{s,p_2}}.
\end{align*}
\end{lemma}

On the other hand, the estimates for handling the derivatives in Euclidean directions are more technical, since the functions in this case are no longer scalar but vector-valued. In particular, the fractional chain rule for H\"older-continuous functions, firstly introduced by Visan \cite{defocusing5dandhigher} and applied in the proof of Lemma \ref{lem3.9}, does in general not hold in the vector-valued setting.

To overcome such issue, our approach is to consider directly derivatives within Besov spaces, the latter space being a more friendly object for dealing with vector-valued functions. For this purpose, we firstly state the following well-known vector-valued fractional chain and product rule due to Kenig, Ponce and Vega \cite{Anisotropic}.

\begin{lemma}[Anisotropic fractional chain rule, \cite{Anisotropic}]\label{lem3.11}
Let $\sigma\in(0,1)$ and $F\in C^{1}(\C)$. Let also the numbers $p,q,p_1,p_2,q_2\in(1,\infty)$ and $q_1\in(1,\infty]$ satisfy $p^{-1}=p_1^{-1}+p_2^{-1}$ and $q^{-1}=q_1^{-1}+q_2^{-1}$. Then for $u:\R^m\times \T^n\to \C$ we have
\begin{align*}
\|D^{\sigma}_x F(u)\|_{L_x^p L_y^q}\lesssim \|F'(u)\|_{L_x^{p_1} L_y^{q_1}} \|D_x^{\sigma} u\|_{L_x^{p_2} L_y^{q_2}}.
\end{align*}
\end{lemma}

\begin{lemma}[Anisotropic fractional product rule, \cite{Anisotropic}]\label{lem3.12}
Let $\sigma\in(0,1)$. Let also the numbers
$$p,p_1,p_2,p_3,p_4,q,q_1, q_2,q_3,q_4\in(1,\infty)$$
satisfy
$p^{-1}=p_1^{-1}+p_2^{-1}=p_3^{-1}+p_4^{-1}$ and $q^{-1}=q_1^{-1}+q_2^{-1}=q_3^{-1}+q_4^{-1}$. Then for $f,g:\R^m\times \T^n\to \C$ we have
\begin{align*}
\|D^{\sigma}_x (fg)\|_{L_x^p L_y^q}\lesssim \|D_x^\sigma f\|_{L_x^{p_1} L_y^{q_1}}\|g\|_{L_x^{p_2} L_y^{q_2}} + \|f\|_{L_x^{p_3} L_y^{q_3}}\|D_x^\sigma g\|_{L_x^{p_4} L_y^{q_4}} .
\end{align*}
\end{lemma}

Before we finally prove the main lemma, we still need the following useful embedding result.

\begin{lemma}\label{lem3.13}
For any $s,\sigma\in(0,\infty)$ and $p,q,r\in(1,\infty)$ satisfying $q^{-1}-p^{-1}=\sigma/m$ we have
\begin{align*}
\|D_x^{s} u\|_{L_x^{p} L_y^r}\lesssim \|u\|_{B_{x,q,q}^{s+\sigma} L_y^r}.
\end{align*}
\end{lemma}

\begin{proof}
First notice that by Lemma \ref{lemallprop} (ii) it holds $\| D_x^s u\|_{L_x^{p} L_y^r}
\lesssim \|D_x^s u\|_{B_{x,q,q}^{\sigma} L_y^r}$. On the other hand, one easily verifies that for a dyadic number $M\in 2^\Z$ the multiplier $P_{\sim M}^x D_x^s$ satisfies the Mikhlin condition \cite[(1.3)]{McConnell} with operator norm $M^s$. Hence by the vector-valued Mikhlin multiplier theorem (see \cite[Thm. 1.1]{McConnell}) we know that $\|P_M^x (D_x^s u)\|_{L_x^p L_y^r}\lesssim M^s \|P_M^x u\|_{L_x^p L_y^r}$, from which the desired claim follows by combining the definition of Besov spaces.
\end{proof}

We are now ready to prove the main lemma.

\begin{lemma}\label{lem3.13+}
Let $\alpha\in(1,\infty)$, $s\in[0,\alpha)$ and $m\in\Z$. Let also $p,q,p_1,p_2,q_1,q_2\in(1,\infty)$ satisfy $p^{-1}=(\alpha-1)p_1^{-1}+p_2^{-1}$ and $q^{-1}=(\alpha-1)q_1^{-1}+q_2^{-1}$. For $i\in\{1,2\}$ let also the numbers $\sigma_a,\sigma_b$ satisfy $\sigma_a\in(0,m(1-p_1^{-1}))$ and $\sigma_b\in(0,m(1-p_2^{-1}))$. Then for $u:\R^m\times \T^n\to \C$ we have
\begin{align*}
\|D_x^s(|u|^{\alpha-k} u^k)\|_{L_x^{p}L_y^{q}}\lesssim \|u\|_{ B_{x,(\frac{1}{p_1}+\frac{\sigma_a}{m})^{-1},(\frac{1}{p_1}+\frac{\sigma_a}{m})^{-1}}^{\sigma_a}L_y^{q_1}}^{\alpha-1}\|u\|
_{B_{x,(\frac{1}{p_2}+\frac{\sigma_b}{m})^{-1},(\frac{1}{p_1}+\frac{\sigma_b}{m})^{-1}}^{s+\sigma_b}L_y^{q_2}}.
\end{align*}
\end{lemma}

\begin{proof}
First, for $s\in(0,1]$ we use Lemma \ref{lem3.11}, Lemma \ref{lemallprop} (ii) and Lemma \ref{lem3.13} to conclude that
\begin{align*}
\begin{aligned}
\|D_x^s(|u|^{\alpha-k} u^k)\|_{L_x^{p}L_y^{q}}&\lesssim
\|u\|_{ L_x^{p_1} L_y^{q_1}}^{\alpha-1}\|D_x^s u\|_{L_x ^{p_2}L_y^{q_2}}\\
&\lesssim \|u\|_{ B_{x,(\frac{1}{p_1}+\frac{\sigma_a}{m})^{-1},(\frac{1}{p_1}+\frac{\sigma_a}{m})^{-1}}^{\sigma_a}L_y^{q_1}}^{\alpha-1}\|u\|
_{B_{x,(\frac{1}{p_2}+\frac{\sigma_b}{m})^{-1},(\frac{1}{p_1}+\frac{\sigma_b}{m})^{-1}}^{s+\sigma_b}L_y^{q_2}}.
\end{aligned}
\end{align*}
Hence we may always assume that $s\in(1,\alpha)$ in the rest of the proof. We shall prove the claim via induction on $\alpha$. Consider first the case $\alpha\in (1,2]$. Using the characterization of $L^p_x L_y^q$-norm via the Littlewood-Paley square function (see e.g. \cite[Lem. A.3]{Anisotropic}) one easily verifies $\|D_x^s(|u|^{\alpha-k} u^k)\|_{L_x^{p}L_y^{q}}\sim \|D_x^{s-1}(\nabla_x(|u|^{\alpha-k} u^k))\|_{L_x^{p}L_y^{q}}$. Let $0<\vare\ll 1$ be a sufficiently small number and define $p_a$, $p_b$, $q_a$, $q_b$ via
\begin{align}\label{3.10}
\begin{aligned}
\frac{1}{p_a}&=\frac{1-\frac{1-\vare}{s}}{p_1}+\frac{\frac{1-\vare}{s}}{p_2}+\frac{\vare}{m},\quad
\frac{1}{p_b}=\frac{1}{p}-\frac{1}{p_a},\\
\frac{1}{q_a}&=\frac{1-\frac{1-\vare}{s}}{q_1}+\frac{\frac{1-\vare}{s}}{q_2},\quad
\frac{1}{q_b}=\frac{1}{q}-\frac{1}{q_a}.
\end{aligned}
\end{align}
Since $1<s<\alpha$, the numbers $p_a$, $p_b$, $q_a$, $q_b$ are all in $(1,\infty)$ by choosing $\vare\ll 1$. Moreover, by direct computation one verifies that
\begin{align*}
\frac{1}{p_b}+\frac{\vare}{m}=\frac{1-\frac{s-1-\vare}{s(\alpha-1)}}{(p_1/(\alpha-1))}
+\frac{\frac{s-1-\vare}{s(\alpha-1)}}{(p_2/(\alpha-1))},\quad
\frac{1}{q_b}=\frac{1-\frac{s-1-\vare}{s(\alpha-1)}}{(q_1/(\alpha-1))}
+\frac{\frac{s-1-\vare}{s(\alpha-1)}}{(q_2/(\alpha-1))}.
\end{align*}
Then using Lemma \ref{lem3.12} and Lemma \ref{lem3.13} we obtain
\begin{align}\label{3.12}
\begin{aligned}
&\,\|D_x^s(|u|^{\alpha-k} u^k)\|_{L_x^{p}L_y^{q}}\sim \|D_x^{s-1}(\nabla_x(|u|^{\alpha-k} u^k))\|_{L_x^{p}L_y^{q}}\\
\lesssim&\, \|u\|_{L_x^{p_1}L_y^{q_1}}^{\alpha-1}\|D_x^{s-1} (\nabla_x u)\|_{L_x^{p_2}L_y^{q_2}}
+\|\nabla_x u\|_{L_x^{p_a}L_y^{q_a}}\|D_x^{s-1}(u^{\alpha-1})\|_{L_x^{p_b}L_y^{q_b}}\\
\lesssim&\,  \|u\|_{ B_{x,(\frac{1}{p_1}+\frac{\sigma_a}{m})^{-1},(\frac{1}{p_1}+\frac{\sigma_a}{m})^{-1}}^{\sigma_a}L_y^{q_1}}^{\alpha-1}\|u\|
_{B_{x,(\frac{1}{p_2}+\frac{\sigma_b}{m})^{-1},(\frac{1}{p_1}+\frac{\sigma_b}{m})^{-1}}^{s+\sigma_b}L_y^{q_2}}\\
&\qquad+\|\nabla_x u\|_{L_x^{p_a}L_y^{q_a}}\|u^{\alpha-1}\|_{B_{x,(\frac{1}{p_b}+\frac{\vare}{m})^{-1},(\frac{1}{p_b}+\frac{\vare}{m})^{-1}}^{s-1+\vare}L_y^{q_b}}.
\end{aligned}
\end{align}
It remains to estimate the second product in \eqref{3.12}. Using Lemma \ref{lem3.13}, Lemma \ref{lemallprop} (vii), \eqref{3.10} and Lemma \ref{lemallprop} (iv) we obtain
\begin{align*}
\begin{aligned}
\|\nabla_x u\|_{L_x^{p_a}L_y^{q_a}}&\lesssim\|u\|_{B_{x,(\frac{1}{p_a}+\frac{\vare}{m})^{-1},(\frac{1}{p_a}+\frac{\vare}{m})^{-1}}^{1+\vare} L_y^{q_a}}\\
&\lesssim
\|u\|_{B_{x,(\frac{1}{p_1}+\frac{2\vare}{m})^{-1},(\frac{1}{p_1}+\frac{2\vare}{m})^{-1}}^{2\vare} L_y^{q_1}}^{1-\frac{1-\vare}{s}}
\|u\|_{B_{x,(\frac{1}{p_2}+\frac{2\vare}{m})^{-1},(\frac{1}{p_2}+\frac{2\vare}{m})^{-1}}^{s+2\vare} L_y^{q_2}}^{\frac{1-\vare}{s}}\\
&\lesssim
\|u\|_{B_{x,(\frac{1}{p_1}+\frac{\sigma_a}{m})^{-1},(\frac{1}{p_1}+\frac{\sigma_a}{m})^{-1}}^{\sigma_a} L_y^{q_1}}^{1-\frac{1-\vare}{s}}
\|u\|_{B_{x,(\frac{1}{p_2}+\frac{\sigma_b}{m})^{-1},(\frac{1}{p_2}+\frac{\sigma_b}{m})^{-1}}^{s+\sigma_b} L_y^{q_2}}^{\frac{1-\vare}{s}}.
\end{aligned}
\end{align*}
In the same way, by also combining Corollary \ref{cor 3.2} we infer that
\begin{align*}
\begin{aligned}
\|u^{\alpha-1}\|_{B_{x,(\frac{1}{p_b}+\frac{\vare}{m})^{-1},(\frac{1}{p_b}+\frac{\vare}{m})^{-1}}^{s-1+\vare}L_y^{q_b}}
&\lesssim \|u\|^{\alpha-1}
_{B_{x,(\alpha-1)(\frac{1}{p_b}+\frac{\vare}{m})^{-1},(\alpha-1)(\frac{1}{p_b}+\frac{\vare}{m})^{-1}}^{\frac{s-1+\vare}{\alpha-1}}L_y^{(\alpha-1)q_b}}\\
&\lesssim
\|u\|_{B_{x,p_1,p_1}^{0} L_y^{q_1}}^{\alpha-1-(1-\frac{1-\vare}{s})}
\|u\|_{B_{x,p_2,p_2}^{s} L_y^{q_2}}^{1-\frac{1-\vare}{s}}
\\
&\lesssim
\|u\|_{B_{x,(\frac{1}{p_1}+\frac{\sigma_a}{m})^{-1},(\frac{1}{p_1}+\frac{\sigma_a}{m})^{-1}}^{\sigma_a} L_y^{q_1}}^{\alpha-1-(1-\frac{1-\vare}{s})}
\|u\|_{B_{x,(\frac{1}{p_2}+\frac{\sigma_b}{m})^{-1},(\frac{1}{p_2}+\frac{\sigma_b}{m})^{-1}}^{s+\sigma_b} L_y^{q_2}}^{1-\frac{1-\vare}{s}}
\end{aligned}
\end{align*}
and the desired claim for $s\in(1,\alpha)$ follows.

Assume now the claim holds for $\alpha\in(1,N]$ with $N\in\N_{\geq 2}$. We aim to show that the claim continues to hold for $\alpha\in(N,N+1]$. For $k\in \N$, $0\leq k\leq \lfloor s\rfloor-1$ and $0<\vare\ll 1$ define
\begin{align*}
\begin{aligned}
\frac{1}{\hat{p}_{a,k}}&=\frac{\frac{k}{s}}{p_1}+\frac{1-\frac{k}{s}}{p_2},\quad
\frac{1}{\hat{p}_{b,k}}=\frac{1}{p}-\frac{1}{\hat{p}_{a,k}},\\
\frac{1}{\hat{q}_{a,k}}&=\frac{\frac{k}{s}}{q_1}+\frac{1-\frac{k}{s}}{q_2},\quad
\frac{1}{\hat{q}_{b,k}}=\frac{1}{q}-\frac{1}{\hat{q}_{a,k}},\\
\frac{1}{p_{a,k}}&=\frac{1-\frac{1+k-\vare}{s}}{p_1}+\frac{\frac{1+k-\vare}{s}}{p_2}+\frac{\vare}{m},\quad
\frac{1}{p_{b,k}}=\frac{1}{p}-\frac{1}{p_{a,k}},\\
\frac{1}{q_{a,k}}&=\frac{1-\frac{1+k-\vare}{s}}{q_1}+\frac{\frac{1+k-\vare}{s}}{q_2},\quad
\frac{1}{q_{b,k}}=\frac{1}{q}-\frac{1}{q_{a,k}}.
\end{aligned}
\end{align*}
Using Leibniz rule and Lemma \ref{lem3.12} we obtain
\begin{align}\label{3.12+}
\begin{aligned}
\|D_x^s (|u|^{\alpha-k} u^k)\|&\lesssim \sum_{k=0}^{\lfloor s\rfloor-1}
\|D_x^{s-1-k}(\nabla_x u)\|_{L_x^{\hat{p}_{a,k}} L_y^{\hat{q}_{a,k}}}\|D_x^k(u^{\alpha-1})\|_{L_x^{\hat{p}_{b,k}} L_y^{\hat{q}_{b,k}}}\\
&\qquad+\sum_{k=0}^{\lfloor s\rfloor-1}
\|D_x^{k}(\nabla_x u)\|_{L_x^{{p}_{a,k}} L_y^{{q}_{a,k}}}\|D_x^{s-1-k}(u^{\alpha-1})\|_{L_x^{{p}_{b,k}} L_y^{{q}_{b,k}}}.
\end{aligned}
\end{align}
We will only prove the claimed estimate for the first summand in \eqref{3.12+}, the one for the second summand in \eqref{3.12+} can be proved by using similar interpolation arguments and the induction hypothesis. Using Lemma \ref{lem3.13} and Lemma \ref{lemallprop} (vii) we first infer that
\begin{align*}
\|D_x^{s-1-k}(\nabla_x u)\|_{L_x^{\hat{p}_{a,k}} L_y^{\hat{q}_{a,k}}}
&\lesssim
\|u\|_{B_{x,(\frac{1}{\hat{p}_{a,k}}+\frac{\vare}{m})^{-1}
(\frac{1}{\hat{p}_{a,k}}+\frac{\vare}{m})^{-1}}^{s-k+\vare}L_y^{\hat{q}_{a,k}}}\\
&\lesssim
\|u\|_{B_{x,p_1,p_1}^{0}L_y^{q_1}}^{\frac{k}{s}}
\|u\|_{B_{x,p_2,p_2}^{s}L_y^{q_2}}^{1-\frac{k}{s}}\\
&\lesssim
\|u\|_{B_{x,(\frac{1}{p_1}+\frac{\sigma_a}{m})^{-1},(\frac{1}{p_1}+\frac{\sigma_a}{m})^{-1}}^{\sigma_a} L_y^{q_1}}^{\frac{k}{s}}
\|u\|_{B_{x,(\frac{1}{p_2}+\frac{\sigma_b}{m})^{-1},(\frac{1}{p_2}+\frac{\sigma_b}{m})^{-1}}^{s+\sigma_b} L_y^{q_2}}^{1-\frac{k}{s}}.
\end{align*}
Next, we first notice that
\begin{align*}
\frac{1}{\hat{p}_{b,k}}&=\frac{\alpha-2}{p_1}+(\frac{1-\frac{k}{s}}{p_1}+\frac{\frac{k}{s}}{p_2})
=:\frac{\alpha-2}{p_1}+\frac{1}{\hat{r}_{p,k}},\\
\frac{1}{\hat{q}_{b,k}}&=\frac{\alpha-2}{q_1}+(\frac{1-\frac{k}{s}}{q_1}+\frac{\frac{k}{s}}{q_2})
=:\frac{\alpha-2}{q_1}+\frac{1}{\hat{r}_{q,k}}.
\end{align*}
Hence using the induction hypothesis and Lemma \ref{lemallprop} (vii) we obtain
\begin{align*}
&\,\|D_x^k(u^{\alpha-1})\|_{L_x^{\hat{p}_{b,k}} L_y^{\hat{q}_{b,k}}}\\
\lesssim&\,
\|u\|_{ B_{x,(\frac{1}{p_1}+\frac{\sigma_a}{m})^{-1},(\frac{1}{p_1}+\frac{\sigma_a}{m})^{-1}}^{\sigma_a}L_y^{q_1}}^{\alpha-2}\|u\|
_{B_{x,(\frac{1}{\hat{r}_{p,k}}+\frac{\vare}{m})^{-1},(\frac{1}{\hat{r}_{p,k}}+\frac{\vare}{m})^{-1}}^{k+\vare}L_y^{\hat{r}_{q,k}}}\\
\lesssim&\,
\|u\|_{ B_{x,(\frac{1}{p_1}+\frac{\sigma_a}{m})^{-1},(\frac{1}{p_1}+\frac{\sigma_a}{m})^{-1}}^{\sigma_a}L_y^{q_1}}^{\alpha-2}
\|u\|_{B_{x,(p_1+\frac{\vare}{m})^{-1},(p_1+\frac{\vare}{m})^{-1}}^{\vare}L_y^{q_1}}^{1-\frac{k}{s}}
\|u\|_{B_{x,(p_2+\frac{\vare}{m})^{-1},(p_2+\frac{\vare}{m})^{-1}}^{s+\vare}L_y^{q_2}}^{\frac{k}{s}}\\
\lesssim&\,
\|u\|_{ B_{x,(\frac{1}{p_1}+\frac{\sigma_a}{m})^{-1},(\frac{1}{p_1}+\frac{\sigma_a}{m})^{-1}}^{\sigma_a}L_y^{q_1}}^{\alpha-1-\frac{k}{s}}
\|u\|_{B_{x,(\frac{1}{p_2}+\frac{\sigma_b}{m})^{-1},(\frac{1}{p_2}+\frac{\sigma_b}{m})^{-1}}^{s+\sigma_b} L_y^{q_2}}^{\frac{k}{s}}
\end{align*}
and the desired estimate follows.
\end{proof}

\subsection{Strichartz estimates}
This subsection is devoted to the proof of the Strichartz estimates which will be used throughout the paper. We firstly record a useful discrete-type interpolation lemma.

\begin{lemma}\label{lem 3.1}
For $\kappa\in \R$ and $\vare\in(0,\infty)$ we have
\begin{align*}
\|f_M\|_{\ell_M^{\kappa,1}(2^{\N_0})}\lesssim
\|f_M\|_{\ell_M^{\kappa-\vare,\infty}(2^{\N_0})}^{\frac12}
\|f_M\|_{\ell_M^{\kappa+\vare,\infty}(2^{\N_0})}^{\frac12}.
\end{align*}
\end{lemma}

\begin{proof}
Let $J:=\|f_M\|_{\ell_M^{\kappa-\vare,\infty}(2^{\N_0})}^{-\frac12}
\|f_M\|_{\ell_M^{\kappa+\vare,\infty}(2^{\N_0})}^{\frac12}$. Then
\begin{align*}
\|f_M\|_{\ell_M^{\kappa,1}(2^{\N_0})}&=\sum_{M\in 2^{\N_0}} M^\kappa|f_M|=\sum_{M\lesssim J} M^\vare (M^{\kappa-\vare}|f_M|)
+\sum_{M\gtrsim J} M^{-\vare} (M^{\kappa+\vare}|f_M|)\\
&\lesssim J^\vare \|f_M\|_{\ell_M^{\kappa-\vare,\infty}(2^{\N_0})}+
J^{-\vare} \|f_M\|_{\ell_M^{\kappa+\vare,\infty}(2^{\N_0})}\\
&=\|f_M\|_{\ell_M^{\kappa-\vare,\infty}(2^{\N_0})}^{\frac12}
\|f_M\|_{\ell_M^{\kappa+\vare,\infty}(2^{\N_0})}^{\frac12}.
\end{align*}
This completes the proof.
\end{proof}

We now state the main Strichartz estimates.

\begin{lemma}\label{prop1}
Let $p\in[2+\frac{4}{m},\infty)$, $\vare\in[0,\frac{n}{2})$, $\sigma=\frac{m}{2}-\frac{m+2}{p}$ and $\alpha\in(0,\frac{1}{p})$. Moreover, let $(p,q)$ be an admissible Strichartz pair. Then for $N_x,N_y\in 2^{\N_0}$ we have
\begin{gather}
\|A^x_{N_x} A^y_{N_y} f\|_{B_{t,p,1}^\alpha L_x^p L_y^{(\frac12-\frac{\vare}{n})^{-1}}(\R\times\Omega)}
\lesssim N_x^{\sigma}N_y^{\vare} N_{\max}^{2\alpha}\|f\|_{Y^0},\label{3.1}\\
\|f\|_{L_t^p L_x^q L_y^2(\R\times\Omega)}
\lesssim \|f\|_{Y^0},\label{3.1+}
\end{gather}
where $A^x_{N_x}\in\{ P_{N_x}^x, P_{<N_x}^x\}$, $A^y_{N_x}\in\{ P_{N_y}^y, P_{<N_y}^y\}$ and $N_{\max}:=\max\{N_x,N_y\}$.
\end{lemma}

\begin{proof}
We only prove the estimate \eqref{3.1} in the case $A^x_{N_x}=P^x_{N_x}$ and $A^y_{N_y}=P^y_{N_y}$, the other cases can be estimated similarly. Recall first the following Strichartz estimate on $\R^m\times\T^n$ (see \cite{TNCommPDE})
\begin{align}
\|P^x_{N_x}\py e^{it\Delta_{x,y}}\phi\|_{L_{t,x}^p L_y^{(\frac{1}{2}-\frac{\vare}{n})^{-1}}(\R\times\Omega)}
&\lesssim N_x^\sigma N_y^{\vare}\|\phi\|_{L_{x,y}^2}\label{strichartz2},\\
\|e^{it\Delta_{x,y}}\phi\|_{L_t^p L_x^q L_y^2}&\lesssim \|\phi\|_{L_{x,y}^2}\label{strichartz1}.
\end{align}
\eqref{3.1+} already follows from \eqref{strichartz1} and the transfer principle (\cite[Prop. 2.19]{HadacHerrKoch2009}). Next, using Bernstein, the identity
\begin{align*}
(i\pt_t)^k P_M^t(e^{it\Delta_{x,y}}f)=P_M^t((-\Delta_{x,y}^k)e^{it\Delta_{x,y}} f)
\end{align*}
and \eqref{strichartz2} we know that
\begin{align}\label{3.2}
\|P_M^t(\px\py e^{it\Delta_{x,y}}\phi)\|_{\ell_{M}^{k,\infty}L_{t,x}^p L_y^{(\frac12-\frac{\vare}{n})^{-1}}(2^{\N_0}\times\R\times\Omega)}
\lesssim N_x^{\sigma}N_y^{\vare} N_{\max}^{2k}\|\phi\|_{L_{x,y}^2}
\end{align}
holds for any $k\in\N_0$, which in fact holds for any $k\in[0,\infty)$ by interpolation. Using Minkowski, \eqref{3.2} and Lemma \ref{lem 3.1} it follows
\begin{equation}\label{3.3}
\begin{aligned}
\|\px\py e^{it\Delta_{x,y}}\phi\|_{B_{t,p,1}^\kappa L_{x}^pL_y^{(\frac12-\frac{\vare}{n})^{-1}}(\R\times\Omega)}
&=
\|P_M^t(\px\py e^{it\Delta_{x,y}}\phi)\|_{\ell_{M}^{\kappa,1} L_{t,x}^p L_y^{(\frac12-\frac{\vare}{n})^{-1}}(2^{\N_0}\times\R\times\Omega)}\\
&\lesssim N_x^{\sigma}N_y^{\vare} N_{\max}^{2\kappa}\|\phi\|_{L_{x,y}^2}
\end{aligned}
\end{equation}
for any $\kappa\in(0,\infty)$. Next, we recall from \cite[Lem. 3.1]{KwakKwon} that for $I=\cup_{j=1}^n I_j$ with $I_j$ consecutive disjoint intervals and any Banach space $E$ it holds
\begin{align}\label{3.5}
\|\sum_{j=1}^n f_j\chi_{I_j} \|_{B_{t,p,\infty}^\alpha E}\lesssim (\sum_{j=1}^n \|f_j\|_{B_{t,p,1}^\alpha E}^p )^{\frac{1}{p}},
\end{align}
where $\chi_{I_j}$ is the sharp cut-off function of $I_j$ and the numerical constant does not depend on the choice of $I_j$. \eqref{3.5}, Minkowski and \eqref{3.3} then yield
\begin{equation}\label{3.6}
\begin{aligned}
&\,\|\sum_{j=1}^n \px\py e^{it\Delta_{x,y}}\phi_j\chi_{I_j} \|_{B_{t,p,\infty}^\alpha  L_{x}^p (\R\times\Omega)}\\
\lesssim&\, \|\px\py e^{it\Delta_{x,y}}\phi_j \|_{\ell_j^p B_{t,p,1}^\alpha  L_{x}^pL_y^{(\frac12-\frac{\vare}{n})^{-1}}(\{1\leq j\leq n\}\times\R\times\Omega)}\\
\lesssim&\,  N_x^{\sigma}N_y^{\vare}N_{\max}^{2\alpha}(\sum_{j=1}^n \|\phi_j\|_{L_{x,y}^2}^p)^{\frac{1}{p}}.
\end{aligned}
\end{equation}
Then \eqref{3.1} follows from \eqref{3.6}, Lemma \ref{lemallprop} (vi), the density of the set of atomic functions in $U_\Delta^p$-space, and the embedding $Y^0\hookrightarrow U_\Delta^p L_{x,y}^2$.
\end{proof}

\subsection{The Galilean transformation}
In the final preliminary subsection, we introduce the Galilean transformation which leaves the NLS-flow invariant. Such invariant property of the Galilean transformation will also play a fundamental role in establishing the multilinear estimates, see Section \ref{sec 5} for details.

For $\xi=(\xi_x,\xi_y)\in\R^m\times\Z^n$ and $z=(x,y)\in\R^{m}\times\T^n$ we define the Galilean transformation $I_\xi$ by
\begin{align*}
I_{\xi}u(t,z)=e^{iz\cdot\xi-it|\xi|^2}u(t,z-2t\xi).
\end{align*}
The following properties of the Galilean transformation can be easily verified by using its definition, see also \cite[Prop. 2.16]{KwakKwon}.

\begin{lemma}\label{lem 3.8}
For any $\xi,\xi_1,\xi_2\in \R^m\times\Z^n$ we have the following properties of the Galilean transformation.
\begin{itemize}
\item $(i\pt_t +\Delta_{x,y})I_\xi u=I_\xi (i\pt_t +\Delta_{x,y}) u$.
\item $I_{\xi_1+\xi_2}u=I_{\xi_1} I_{\xi_2} u$.
\item For any set $C\subset \R^m\times\Z^n$ we have $P_{C+\xi}I_\xi u= I_\xi P_C u$.
\item $\|I_\xi u\|_{Y^0}=\|u\|_{Y^0}$.
\end{itemize}
\end{lemma}

\section{The $Z^s$-space}
For $s\in\R$ we define the $Z^s$-norm by
\begin{equation}\label{eq 4.1}
\begin{aligned}
\|u\|_{Z^s}&:=\max_{p\in\max\{m^*,\frac{1}{\sigma_1}\}}\max_{\alpha\in\max\{\sigma,\frac{1}{p}-\sigma\}}\bg(\|P_R u\|_{\ell_R^{s,2}L_t^p L_x^q L_y^2}
+\|P_R u\|_{\ell_R^{s-2\alpha,2}B_{t,p,2}^{\alpha} L_x^q L_y^2}\\
&\qquad+\|\max_{R\ll N}R^{-4\sigma}\|P_{\leq 8R}I_{Rk} P_N u\|_{\ell_k^2 B_{t,p,2}^{2\sigma} L_x^{q} L_y^{2}}\|_{\ell^{s,2}_N}\bg),
\end{aligned}
\end{equation}
where
\begin{align}\label{strichartz condition}
m^*=\left\{
\begin{array}{ll}
(\frac{1}{4}-\sigma)^{-1},&\text{when $m=1$,}\\
(\frac{1}{2}-\sigma)^{-1},&\text{when $m=2$,}\\
2,&\text{when $m\geq 3$}
\end{array}
\right.
\end{align}
and $(p,q)$ is an admissible Strichartz pair. The spaces $Y^s$ and $Z^s$ will be the main function spaces for establishing the small scattering results. The following lemma reveals the embedding relationship between the both spaces.

\begin{lemma}\label{lem 4.1}
For any $s\in\R$ we have $Y^s\hookrightarrow Z^s$.
\end{lemma}

\begin{proof}
By Lemma \ref{prop1} we infer that for any $\alpha\in(0,\infty)$ and admissible Strichartz pair $(p,q)$ we have
\begin{equation*}
\begin{aligned}
\|P_R u\|_{\ell_R^{s,2} L_t^p L_x^q L_y^2}+
\|P_R u\|_{\ell_R^{s-2\alpha,2}B_{t,p,2}^{\alpha} L_x^q L_y^2}&\lesssim\|P_R u\|_{\ell_R^{s,2}Y^0}=\|P_R u\|_{Y^s},
\end{aligned}
\end{equation*}
which implies the estimate for the first norm. For the second norm, we use \eqref{3.1} and Lemma \ref{lem 3.8} to obtain
\begin{equation*}
\begin{aligned}
&\,\|\max_{R\ll N}R^{-4\sigma}\|P_{\leq 8R}I_{Rk} P_N u\|_{\ell_k^2 B_{t,p,2}^{2\sigma} L_x^{q} L_y^{2}}\|_{\ell^{s,2}_N}\\
\lesssim&\,\|\max_{R\ll N}\|P_{\leq 8R}I_{Rk} P_N u\|_{\ell_k^2 Y^0}\|_{\ell^{s,2}_N}\\
=&\,
\|\max_{R\ll N}\|I_{Rk} P_{-Rk+[-8R,8R]^d} P_N u\|_{\ell_k^2 Y^0}\|_{\ell^{s,2}_N}\\
=&\,\|\max_{R\ll N}\|P_{-Rk+[-8R,8R]^d} P_N u\|_{\ell_k^2 Y^0}\|_{\ell^{s,2}_N}\\
\lesssim&\, \|P_N u\|_{\ell^{s,2}_NY^0}\sim \|u\|_{Y^s}.
\end{aligned}
\end{equation*}
This completes the proof.
\end{proof}

In the rest of the paper, we fix the number $s$ to be $s=s_c=\frac{d}{2}-\frac{2}{\alpha}$.

\section{Multilinear estimates}\label{sec 5}
In this section we aim to estimate the dual norm $\|v^* |u|^{\alpha-k} u^k\|_{(Z^0)'}$ with $v^*\in\{v,\bar v\}$ and $m\in\Z$. This will be done by considering an appropriate estimate of the integral $\int A|u|^2\,dxdydt$. By dyadic decomposition and symmetry of $u$ and $\bar u$ it suffices to consider
\begin{itemize}
\item[(I)]  $\sum_{R\lesssim N\sim  M}\int P_R u \,P_N \bar u \, P_M A\,dxdydt$ and
\item[(II)] $\sum_{R\ll N\sim  M}\int P_R A\,P_N u \,P_M \bar u  \,dxdydt$.
\end{itemize}
The rest of this section is devoted to estimating the terms I and II.

\subsection{Estimation in Regime I}
\begin{lemma}\label{est_lem_1}
For any $\sigma>0$ we have
\begin{equation*}
\begin{aligned}
&\,\sum_{R\lesssim N\sim  M}\int P_R u \,P_N \bar u \, P_M A\,dxdydt\\
\lesssim &\,\|u\|_{Z^0}^2 (\|A\|_{L_t^{(p_0/2)'} B_{x,(\frac{2}{m+2}+\frac{\sigma}{m})^{-1},\infty}^{2\sigma} L_y^{\frac{n}{\sigma}}}
+
\|A\|_{L_{t,x}^{(p_0/2)'}B_{y,\frac{n}{2\sigma},\infty}^{2\sigma}}).
\end{aligned}
\end{equation*}
\end{lemma}

\begin{proof}
Using H\"older it follows
\begin{align*}
&\,\sum_{R\lesssim N\sim  M} \int P_R u \,P_N \bar u \, P_M A\,dxdydt\\
\lesssim&\, \sum_{R\lesssim N\sim  M} \|P_R u\|_{L_{t,x}^{p_0}L_y^{(\frac12-\frac{\sigma}{n})^{-1}}}
\|P_N u\|_{L_{t,x}^{p_0}L_y^{(\frac12-\frac{\sigma}{n})^{-1}}}\|P_M A\|_{L_{t,x}^{p_0}L_y^{\frac{n}{2\sigma}}}\\
\lesssim&\, \sum_{R\lesssim N\sim  M} (R/N)^\sigma R^{-\sigma}\|P_R u\|_{L_{t,x}^{p_0}L_y^2}
N^{-\sigma}\|P_N u\|_{L_{t,x}^{p_0}L_y^2}\\
&\qquad\qquad\times\left(M^{2\sigma}(\|P_M^x P_{<M}^y A\|_{L_{t,x}^{p_0}L_y^{\frac{n}{2\sigma}}}
+\|P_{<M}^x P_M^y A\|_{L_{t,x}^{p_0}L_y^{\frac{n}{2\sigma}}})\right).
\end{align*}
Combining Schur's test, Sobolev, Lemma \ref{lem 4.1} and Minkowski we obtain
\begin{align*}
&\,\sum_{R\lesssim N\sim  M} \int P_R u \,P_N \bar u \, P_M A\,dxdydt\\
\lesssim&\,\|P_R u\|_{\ell_R^{-\sigma,2} L_{t,x}^{p_0}L_y^{(\frac12-\frac{\sigma}{n})^{-1}}}^2
(\|M^{2\sigma}P_M^x A\|_{\ell_M^\infty L_{t,x}^{(p_0/2)'}L_y^{\frac{n}{2\sigma}}}
+\|M^{2\sigma}P_M^y A\|_{\ell_M^\infty L_{t,x}^{(p_0/2)'}L_y^{\frac{n}{2\sigma}}})\\
\lesssim &\, \|P_R u\|_{\ell_R^2 L_{t,x}^{p_0}L_y^{2}}^2(\|A\|_{L_t^{(p_0/2)'} B_{x,(p_0/2)',\infty}^{2\sigma} L_y^{\frac{n}{2\sigma}}}
+\|A\|_{L_{t,x}^{(p_0/2)'} B_{y,\frac{n}{2\sigma},\infty}^{2\sigma}})\\
\lesssim&\, \|u\|_{Z^0}^2(\|A\|_{L_t^{(p_0/2)'} B_{x,(p_0/2)',\infty}^{2\sigma} L_y^{\frac{n}{2\sigma}}}
+\|A\|_{L_{t,x}^{(p_0/2)'} B_{y,\frac{n}{2\sigma},\infty}^{2\sigma}}),
\end{align*}
which completes the proof.
\end{proof}

\subsection{Estimation in Regime II}
Recall that for $\xi=(\xi_x,\xi_y)\in\R^m\times\Z^n$ and $z=(x,y)\in\R^{m}\times\T^n$ the Galilean transformation $I_\xi$ is defined by
\begin{align*}
I_{\xi}u(t,z)=e^{iz\cdot\xi-it|\xi|^2}u(t,z-2t\xi).
\end{align*}
Following the same lines as in \cite{KwakKwon}, we derive suitable multilinear estimates for dealing the sums in Regime II. This will be done by considering the operator \begin{align}\label{def of J}
J_{\xi}A(t,z):=A(t,z-2t\xi)
\end{align}
for $z,\xi\in \R^m\times \T^n$, which is nothing else but the spatial translation appearing in the Galilean transformation. To handle the anisotropic nature of the model, we firstly prove the following crucial estimate which gives an anisotropic generalization of \cite[Lem. 3.8]{KwakKwon}.

\begin{lemma}\label{lem 5.2}
Let $\Omega$ be the open tetrahedron whose vertices are $(\frac12,1,\frac18)$, $(0,1,0)$, $(1,1,0)$ and $(0,0,0)$. Let $(p,r,\rho)$ be a triple of parameters such that $(\frac{1}{p},\frac{1}{r},\rho)$ lies in $\Omega$. Then for $N\in 2^\N$ and $k=(k_x,k_y)\in\Z^{m}\setminus\{0\}\times \Z^n$ we have
\begin{align}\label{5.4}
\|J_{Nk}P_{<N}^x A\|_{B_{t,p,\infty}^{-\rho} L_x^p L_y^{\frac{1}{4\rho}}}
\lesssim (N^{-2\rho}+|k_x|^{-\rho})N^{m(\frac{1}{r}-\frac{1}{p})-2\rho}\|A\|_{B_{t,r,r}^{\frac{1}{r}-\frac{1}{p}}L_x^r L_y^{\frac{1}{4\rho}}}.
\end{align}
Similarly, for $k=(k_x,k_y)\in\Z^{m}\times \Z^n\setminus\{0\}$ we have
\begin{align}\label{5.5}
\|J_{Nk}P_{<N}^y A\|_{B_{t,p,\infty}^{-\rho} L_x^r L_y^{\frac{1}{4\rho}}}
\lesssim (N^{-2\rho}+|k_y|^{-\rho})N^{4\rho n-2\rho}\|A\|_{B_{t,(\frac{1}{r}+4\rho)^{-1},(\frac{1}{r}+4\rho)^{-1}}
^{\frac{1}{r}-\frac{1}{p}+4\rho}L_x^r L_y^{\frac{1}{8\rho}}}.
\end{align}
\end{lemma}

\begin{proof}
We recall the following estimate from \cite[Lem. 3.7]{KwakKwon}\footnote{We note that the original version of \eqref{5.6} is stated under the condition where $P_{<N}^x$ is replaced by $P_N^x$ and the spatial domain $\R$ replaced by $\T$. Similar arguments are applicable under the setting of Lemma \ref{lem 5.2} via a straightforward modification of those given in the proof of \cite[Lem. 3.7]{KwakKwon}.}
\begin{align}\label{5.6}
\|J^x_{Nk_x}P_{<N}^x A\|_{B_{t,2,2}^{-\frac18} L_{x}^2}
\lesssim N^{\frac{m}{2}-\frac14}(N^{-\frac14}+|k_x|^{-\frac18})\|A\|_{B_{t,1,1}^{\frac12}L_x^1},
\end{align}
where $J^x_\xi$ is understood as the translation operator \eqref{def of J} on $\R^m$. Using $J_{Nk}P_{<N}^x=J_{Nk_x}^x P_{<N}^x J_{Nk_y}^y$, the fact that $J_{Nk_y}^y$ leaves the $L_y^p$-norm invariant, taking $L_y^2$-norm on both sides of \eqref{5.6} and applying Minkowski we obtain
\begin{align}\label{5.7}
\|J_{Nk}P_{<N}^x A\|_{B_{t,2,\infty}^{-\frac18} L_{x,y}^2}
\lesssim N^{\frac{m}{2}-\frac14}(N^{-\frac14}+|k_x|^{-\frac18})\|A\|_{B_{t,1,1}^{\frac12}L_x^1L_y^2}.
\end{align}
On the other hand, for $1<r<p<\infty$ we have
\begin{align}\label{5.8}
\|J_{Nk}P_{<N}^x A\|_{B_{t,p,\infty}^{0} L_{x}^p L_y^\infty}\lesssim \|J_{Nk}P_{<N}^x A\|_{L_{t,x}^p L_y^\infty}
\lesssim \|P_{<N}^x A\|_{L_{t,x}^p L_y^\infty}\lesssim N^{m(\frac{1}{r}-\frac{1}{p})}\|A\|_{B_{t,r,r}^{\frac{1}{r}-\frac{1}{p}} L_x^r L_y^\infty}.
\end{align}
Interpolating \eqref{5.7} and \eqref{5.8} yields \eqref{5.4}. Similarly, \eqref{5.6} follows from interpolating
\begin{equation*}
\begin{aligned}
\|J_{Nk}P_{<N}^y A\|_{B_{t,2,\infty}^{-\frac18} L_{x,y}^2}\lesssim
N^{\frac{n}{2}-\frac14}(N^{-\frac14}+|k_y|^{-\frac18})\|A\|_{B_{t,1,1}^{\frac12} L_x^2 L_y^1}
\end{aligned}
\end{equation*}
and
\begin{align*}
\|J_{Nk}P_{<N}^y A\|_{B_{t,p,\infty}^{0}  L_x^r L_y^\infty}
\lesssim \|A\|_{B_{t,r,r}^{\frac{1}{r}-\frac{1}{p}}L_x^r L_y^\infty}.
\end{align*}
This completes the proof.
\end{proof}

We are now ready to give the main estimate in Regime II.

\begin{lemma}\label{est_lem_2}
We have
\begin{equation*}
\begin{aligned}
&\,\sum_{R\ll N\sim  M}\int P_R A\,P_N u \,P_M \bar u  \,dxdydt\\
\lesssim&\,\|u\|_{Z^0}^2 \bg(\|A\|_{B_{t,(\frac{2}{m+2}+4\sigma)^{-1},(\frac{2}{m+2}+4\sigma)^{-1}}^{\sigma}
B_{x,(\frac{2}{m+2}+\frac{4\sigma}{m})^{-1},\infty}^{14\sigma}
L_y^{\frac{n}{4\sigma}}}\\
&+\|A\|_{B_{t,(\frac{2}{m+2}+4\sigma)^{-1},(\frac{2}{m+2}+4\sigma)^{-1}}^{\sigma}
L_x^{(\frac{2}{m+2}+\frac{4\sigma}{m})^{-1}}
B_{y,\frac{n}{4\sigma},\infty}^{14\sigma}}\\
&+\|A\|_{B_{t,(\frac{2}{m+2}+8\sigma)^{-1},(\frac{2}{m+2}+8\sigma)^{-1}}^{5\sigma}
B_{x,(\frac{2}{m+2}+\frac{4\sigma}{m})^{-1},\infty}^{18\sigma}
L_y^{\frac{n}{8\sigma}}}\\
&+
\|A\|_{B_{t,(\frac{2}{m+2}+8\sigma)^{-1},(\frac{2}{m+2}+8\sigma)^{-1}}^{5\sigma}
L_x^{(\frac{2}{m+2}+\frac{4\sigma}{m})^{-1}}
B_{y,\frac{n}{8\sigma},\infty}^{18\sigma}}
\bg).
\end{aligned}
\end{equation*}
\end{lemma}

\begin{proof}
Using the same computation given in the proof of \cite[Lem. 3.9]{KwakKwon} we arrive at the identity
\begin{equation*}
\begin{aligned}
\int P_R A\,P_N u \,P_M \bar u  \,dxdydt
=\sum_{k\in\Z^d}\int P_{\leq 8R}I_{2Rk} P_N u\overline{P_{(-R,R]^d}I_{2Rk} P_M u}
J_{2R k} P_R A \,dxdydt.
\end{aligned}
\end{equation*}
Notice that $P_{\leq 8R}I_{2Rk} u_N$ is nonzero only if $|k|\gtrsim N/R$. We further discuss two cases: $|k_x|\geq |k|$ and $|k_y|\geq |k|$. In the first case, we use H\"older, Bernstein, \eqref{5.4}, Lemma \ref{kwak_lem_lhh}, \eqref{eq 4.1}, the fact $P_R A=P_{<8R}^x P_R A$ and Cauchy-Schwarz in $k$ to obtain
\begin{equation*}
\begin{aligned}
\lesssim&\,\sum_{\substack{k\in\Z^d\\ |k_x|\geq |k|}}\|P_{\leq 8R}I_{2Rk} P_N u\|_{B_{t,p_0,2}^{2\sigma} L_x^{p_0} L_y^{(\frac12-\frac{2\sigma}{n})^{-1}}}
\|P_{\leq 8R}I_{2Rk} P_M u\|_{B_{t,p_0,2}^{2\sigma} L_x^{p_0} L_y^{(\frac12-\frac{2\sigma}{n})^{-1}}}\\
&\,\qquad\qquad\times\|J_{2Rk}P_{<R}^x P_R A\|_{B_{t,(\frac{2}{m+2}+3\sigma)^{-1},\infty}^{-\sigma} L_x^{(p_0/2)'} L_y^{\frac{n}{4\sigma}}}\\
\lesssim&\,\sum_{\substack{k\in\Z^d\\ |k_x|\geq |k|}}(R^{-4\sigma}\|P_{\leq 8R}I_{2Rk} P_N u\|_{B_{t,p_0,2}^{2\sigma} L_x^{p_0} L_y^{2}})
(R^{-4\sigma}\|P_{\leq 8R}I_{2Rk} P_M u\|_{B_{t,p_0,2}^{2\sigma} L_x^{p_0} L_y^{2}})\\
&\,\qquad\qquad\times R^{15\sigma}\|J_{2Rk}P_{<8R}^x P_R A\|_{B_{t,(\frac{2}{m+2}+3\sigma)^{-1},\infty}^{-\sigma} L_x^{(\frac{2}{m+2}+\frac{3\sigma}{m})^{-1}} L_y^{\frac{n}{4\sigma}}}\\
\lesssim&\,\|P_N u\|_{Z^0}\|P_M u\|_{Z^0} (R^{-2\sigma}+(R/N)^{\sigma})\\
&\qquad\qquad\times R^{14\sigma}
\|P_R A\|_{B_{t,(\frac{2}{m+2}+4\sigma)^{-1},(\frac{2}{m+2}+4\sigma)^{-1}}^{\sigma}L_x^{(\frac{2}{m+2}+\frac{4\sigma}{m})^{-1}} L_y^{\frac{n}{4\sigma}}}\\
\lesssim&\, (R^{-2\sigma}+(R/N)^{\sigma})\|P_N u\|_{Z^0}\|P_M u\|_{Z^0}
 \bg(\|A\|_{B_{t,(\frac{2}{m+2}+4\sigma)^{-1},(\frac{2}{m+2}+4\sigma)^{-1}}^{\sigma}
B_{x,(\frac{2}{m+2}+\frac{4\sigma}{m})^{-1},\infty}^{14\sigma}
L_y^{\frac{n}{4\sigma}}}\\
&\qquad\qquad+
\|A\|_{B_{t,(\frac{2}{m+2}+4\sigma)^{-1},(\frac{2}{m+2}+4\sigma)^{-1}}^{\sigma}
L_x^{(\frac{2}{m+2}+\frac{4\sigma}{m})^{-1}}
B_{y,\frac{n}{4\sigma},\infty}^{14\sigma}}\bg),
\end{aligned}
\end{equation*}
where in the last step we used \eqref{2.1} to estimate $P_R A$ separately. The contribution of the first part follows from firstly summing over $R\ll N$ for fixed $N$ and then applying Cauchy-Schwarz w.r.t. $N$ and $M$. For the case $|k_y|\geq |k|$, we use the following estimate:
\begin{equation*}
\begin{aligned}
\lesssim&\,\sum_{\substack{k\in\Z^d\\ |k_y|\geq |k|}}\|P_{\leq 8R}I_{2Rk} P_N u\|_{B_{t,p_0,2}^{2\sigma} L_x^{(p_0^{-1}-\frac{2\sigma}{m})^{-1}} L_y^{(\frac12-\frac{2\sigma}{n})^{-1}}}\\
&\,\qquad\qquad\times\|P_{\leq 8R}I_{2Rk} P_M u\|_{B_{t,p_0,2}^{2\sigma} L_x^{(p_0^{-1}-\frac{2\sigma}{m})^{-1}} L_y^{(\frac12-\frac{2\sigma}{n})^{-1}}}\\
&\,\qquad\qquad\times\|J_{2Rk}P_{<R}^y P_R A\|_{B_{t,(\frac{2}{m+2}+3\sigma)^{-1},\infty}^{-\sigma} L_x^{(\frac{2}{m+2}+\frac{4\sigma}{m})^{-1}} L_y^{\frac{n}{4\sigma}}}\\
\lesssim&\,\sum_{\substack{k\in\Z^d\\ |k_y|\geq |k|}}(R^{-4\sigma}\|P_{\leq 8R}I_{2Rk} P_N u\|_{B_{t,p_0,2}^{2\sigma} L_x^{p_0} L_y^{2}})
(R^{-4\sigma}\|P_{\leq 8R}I_{2Rk} P_M u\|_{B_{t,p_0,2}^{2\sigma} L_x^{p_0} L_y^{2}})\\
&\,\qquad\qquad\times R^{16\sigma}\|J_{2Rk}P_{<R}^y P_R A\|_{B_{t,(\frac{2}{m+2}+3\sigma)^{-1},\infty}^{-\sigma} L_x^{(\frac{2}{m+2}+\frac{4\sigma}{m})^{-1}}L_y^{\frac{n}{4\sigma}}}\\
\lesssim&\,\|P_N u\|_{Z^0}\|P_M u\|_{Z^0} (R^{-2\sigma}+(R/N)^{\sigma})\\
&\qquad\qquad\times R^{18\sigma}
\|P_R A\|_{B_{t,(\frac{2}{m+2}+8\sigma)^{-1},(\frac{2}{m+2}+8\sigma)^{-1}}^{5\sigma}L_x^{(\frac{2}{m+2}+\frac{4\sigma}{m})^{-1}} L_y^{\frac{n}{8\sigma}}}\\
\lesssim&\, (R^{-2\sigma}+(R/N)^{\sigma})\|P_N u\|_{Z^0}\|P_M u\|_{Z^0}
 \bg(\|A\|_{B_{t,(\frac{2}{m+2}+8\sigma)^{-1},(\frac{2}{m+2}+8\sigma)^{-1}}^{5\sigma}
B_{x,(\frac{2}{m+2}+\frac{4\sigma}{m})^{-1},\infty}^{18\sigma}
L_y^{\frac{n}{8\sigma}}}\\
&\qquad\qquad+
\|A\|_{B_{t,(\frac{2}{m+2}+8\sigma)^{-1},(\frac{2}{m+2}+8\sigma)^{-1}}^{5\sigma}
L_x^{(\frac{2}{m+2}+\frac{4\sigma}{m})^{-1}}
B_{y,\frac{n}{8\sigma},\infty}^{18\sigma}}\bg).
\end{aligned}
\end{equation*}
This completes the proof.
\end{proof}

Summing up, Lemma \ref{est_lem_1} and Lemma \ref{est_lem_2} imply immediately the following lemma, which being a fundamental step in proving the crucial multilinear estimate given in next subsection.
\begin{lemma}\label{lem 5.4}
We have
\begin{equation*}
\begin{aligned}
\||u|^2 A\|_{L_{t,x,y}^1}&\lesssim \|u\|_{Z^0}^2\bg(\|A\|_{L_t^{(p_0/2)'} B_{x,(\frac{2}{m+2}+\frac{\sigma}{m})^{-1},\infty}^{2\sigma} L_y^{\frac{n}{\sigma}}}
+
\|A\|_{L_{t,x}^{(p_0/2)'}B_{y,\frac{n}{2\sigma},\infty}^{2\sigma}}\\
&+\|A\|_{B_{t,(\frac{2}{m+2}+4\sigma)^{-1},(\frac{2}{m+2}+4\sigma)^{-1}}^{\sigma}
B_{x,(\frac{2}{m+2}+\frac{4\sigma}{m})^{-1},\infty}^{14\sigma}
L_y^{\frac{n}{4\sigma}}}\\
&+\|A\|_{B_{t,(\frac{2}{m+2}+4\sigma)^{-1},(\frac{2}{m+2}+4\sigma)^{-1}}^{\sigma}
L_x^{(\frac{2}{m+2}+\frac{4\sigma}{m})^{-1}}
B_{y,\frac{n}{4\sigma},\infty}^{14\sigma}}\\
&+\|A\|_{B_{t,(\frac{2}{m+2}+8\sigma)^{-1},(\frac{2}{m+2}+8\sigma)^{-1}}^{5\sigma}
B_{x,(\frac{2}{m+2}+\frac{4\sigma}{m})^{-1},\infty}^{18\sigma}
L_y^{\frac{n}{8\sigma}}}\\
&+
\|A\|_{B_{t,(\frac{2}{m+2}+8\sigma)^{-1},(\frac{2}{m+2}+8\sigma)^{-1}}^{5\sigma}
L_x^{(\frac{2}{m+2}+\frac{4\sigma}{m})^{-1}}
B_{y,\frac{n}{8\sigma},\infty}^{18\sigma}}\bg).
\end{aligned}
\end{equation*}
\end{lemma}

\subsection{Main multilinear estimate}
In this subsection, we use Lemma \ref{lem 5.4} to prove our main multilinear estimate.
\begin{lemma}\label{lem5.5-}
For $\alpha>\frac{4}{m}$ and $k\in\Z$ we have
\begin{align*}
\|vw|u|^{\alpha-k} u^k\|_{L_{t,x,y}^1}\lesssim \|v\|_{Z^0}\|w\|_{Z^0}\|u\|_{Z^s}^\alpha.
\end{align*}
\end{lemma}

\begin{proof}
Using Lemma \ref{lem 5.4} we know that
\begin{equation*}
\begin{aligned}
\|uA\|_{L_{t,x,y}^2}^2&\lesssim \|u\|_{Z^0}^2\bg(\||A|^2\|_{L_t^{(p_0/2)'} B_{x,(\frac{2}{m+2}+\frac{\sigma}{m})^{-1},\infty}^{2\sigma} L_y^{\frac{n}{\sigma}}}
+
\||A|^2\|_{L_{t,x}^{(p_0/2)'}B_{y,\frac{n}{2\sigma},\infty}^{2\sigma}}\\
&+\||A|^2\|_{B_{t,(\frac{2}{m+2}+4\sigma)^{-1},(\frac{2}{m+2}+4\sigma)^{-1}}^{\sigma}
B_{x,(\frac{2}{m+2}+\frac{4\sigma}{m})^{-1},\infty}^{14\sigma}
L_y^{\frac{n}{4\sigma}}}\\
&+\||A|^2\|_{B_{t,(\frac{2}{m+2}+4\sigma)^{-1},(\frac{2}{m+2}+4\sigma)^{-1}}^{\sigma}
L_x^{(\frac{2}{m+2}+\frac{4\sigma}{m})^{-1}}
B_{y,\frac{n}{4\sigma},\infty}^{14\sigma}}\\
&+\||A|^2\|_{B_{t,(\frac{2}{m+2}+8\sigma)^{-1},(\frac{2}{m+2}+8\sigma)^{-1}}^{5\sigma}
B_{x,(\frac{2}{m+2}+\frac{4\sigma}{m})^{-1},\infty}^{18\sigma}
L_y^{\frac{n}{8\sigma}}}\\
&+
\||A|^2\|_{B_{t,(\frac{2}{m+2}+8\sigma)^{-1},(\frac{2}{m+2}+8\sigma)^{-1}}^{5\sigma}
L_x^{(\frac{2}{m+2}+\frac{4\sigma}{m})^{-1}}
B_{y,\frac{n}{8\sigma},\infty}^{18\sigma}}\bg).
\end{aligned}
\end{equation*}
Using Corollary \ref{cor 3.6} we infer that
\begin{equation*}
\begin{aligned}
\|uA\|_{L_{t,x,y}^2}^2
\lesssim&\,\|u\|_{Z^0}^2\bg(\|A\|^2_{L_t^{m+2} B_{x,(\frac{1}{m+2}+\frac{5\sigma}{2m})^{-1},\infty}^{3\sigma}L_y^{\frac{2n}{\sigma}}}
+\|A\|^2_{L_{t,x}^{m+2} B_{y,\frac{n}{3\sigma},\infty}^{3\sigma}}\\
&\qquad+\|A\|^2_{B_{t,(\frac{1}{m+2}+\frac{7\sigma}{2})^{-1},(\frac{1}{m+2}+2\sigma)^{-1}}^{2\sigma}
B_{x,(\frac{1}{m+2}+\frac{10\sigma}{m})^{-1},\infty}^{15\sigma}
L_y^{\frac{n}{2\sigma}}}\\
&\qquad+\|A\|^2_{B_{t,(\frac{1}{m+2}+\frac{7\sigma}{2})^{-1},(\frac{1}{m+2}+2\sigma)^{-1}}^{2\sigma}
L_x^{(\frac{1}{m+2}+\frac{2\sigma}{m})^{-1}}
B_{y,\frac{n}{10\sigma},\infty}^{15\sigma}}\\
&\qquad+\|A\|^2_{B_{t,(\frac{1}{m+2}+\frac{15\sigma}{2})^{-1},(\frac{1}{m+2}+4\sigma)^{-1}}^{6\sigma}
B_{x,(\frac{1}{m+2}+\frac{12\sigma}{m})^{-1},\infty}^{19\sigma}
L_y^{\frac{n}{4\sigma}}}\\
&\qquad+\|A\|^2_{B_{t,(\frac{1}{m+2}+\frac{15\sigma}{2})^{-1},(\frac{1}{m+2}+4\sigma)^{-1}}^{6\sigma}
L_x^{(\frac{1}{m+2}+\frac{2\sigma}{m})^{-1}}
B_{y,\frac{n}{14\sigma},\infty}^{19\sigma}}\bg)\\
=:&\,\|u\|_{Z^0}^2(I^2+II^2+III^2+IV^2+V^2+VI^2)\\
=:&\,\|u\|_{Z^0}^2\,\mathcal{M}(A)^2.
\end{aligned}
\end{equation*}
Hence
\begin{align*}
\|vw |u|^{\alpha-k} u^k\|_{L_{t,x,y}^1}&\lesssim \|v^* |u|^{\frac{\alpha}{2}-k} u^k \|_{L_{t,x,y}^2}
\|w |u|^{\frac{\alpha}{2}} \|_{L_{t,x,y}^2}
\lesssim \|v\|_{Z^0} \|w\|_{Z^0}{ \mathcal{M}(|u|^{\frac{\alpha}{2}-k} u^k)
\mathcal{M}(|u|^{\frac{\alpha}{2}})},
\end{align*}
which implies $\|v |u|^{\alpha-k} u^k\|_{L_{t,x,y}^1}\lesssim \|v\|_{Z^0}{\mathcal{M}(|u|^{\frac{\alpha}{2}-k} u^k)
\mathcal{M}(|u|^{\frac{\alpha}{2}})}$. To prove the desired claim, it suffices to show that
$ \mathcal{M}(|u|^{\frac{\alpha}{2}-k} u^k)
\lesssim \|u\|_{Z^s}^{\frac{\alpha}{2}}$ holds for any $k\in\Z$. The terms $I$ to $VI$ will be estimated in order. Since $\frac{\alpha}{2}$ is not necessarily a small number, we need a further decomposition in order to guarantee the H\"older continuity of the nonlinear term. Let $\alpha_j\in (0,1)$ and $k_j\in \Z$ ($j=1,...,\zeta$ with some $\zeta\in\N$) satisfy $\sum_{j=1}^\zeta \alpha_j=\frac{\alpha}{2}$ and $\sum_{j=1}^\zeta k_j=k$. Let also $s_x,s_y$ satisfy $1\gg s_x,s_y\gg\sigma$ and define $r_x,r_y$ by $\frac{\alpha}{2}(\frac{1}{r_x}-\frac{s_x}{m})=((\frac{1}{m+2}+\frac{5\sigma}{2m})^{-1})^{-1}-\frac{3\sigma}{m}$ and
$\frac{\alpha}{2}(\frac{1}{r_y}-\frac{s_y}{n})=(\frac{n}{3\sigma})^{-1}-\frac{3\sigma}{n}$. An inductive application of Lemma \ref{kwak_lem_lhh} and Corollary \ref{cor 3.2} yields
\begin{align}\label{5.19}
\begin{aligned}
I+II&\lesssim \prod_{j=1}^\zeta \||u|^{\alpha_j-k_j} u^{k_j}\|_{L_t^{\frac{\alpha(m+2)}{2\alpha_j}} B^{s_x \alpha_j}_{x,r_x/\alpha_j,\infty} L_y^{\frac{\alpha n}{\sigma \alpha_j}}
\,\cap\,
L_{t,x}^{\frac{\alpha(m+2)}{2\alpha_j}} B^{s_y \alpha_j}_{y,r_y/\alpha_j,\infty}}\\
&\lesssim \|u\|_{L_t^{\frac{\alpha(m+2)}{2}} B^{s_x}_{x,r_x,r_x} L_y^{\frac{\alpha n}{\sigma}}
\,\cap\,
L_{t,x}^{\frac{\alpha(m+2)}{2}} B^{s_y}_{y,r_y,r_y}}.
\end{aligned}
\end{align}
By direct computation one easily verifies that $(\frac{\alpha(m+2)}{2},r_x)$ and $(\frac{\alpha(m+2)}{2},\frac{\alpha(m+2)}{2})$ are $(s-\frac{n}{2}-s_x+\frac{\sigma}{\alpha})$- and $(s-\frac{n}{2})$-admissible Strichartz pair respectively. Now let $q_\alpha$ be the number such that $(\frac{\alpha(m+2)}{2},q_\alpha)$ is Strichartz admissible. Then combining Minkowski, Bernstein and Lemma \ref{lemallprop} it follows
\begin{align}\label{5.20}
\begin{aligned}
I+II&\lesssim \| P_{N_x}^x P_{N_y}^y u\|^{\frac{\alpha}{2}}_{\bg(\ell_{N_x}^{(s-\frac{n}{2}-s_x+\frac{\sigma}{\alpha})+s_x,2}
\ell_{N_y}^{\frac{n}{2}-\frac{\sigma}{\alpha}}\,\cap\,
\ell_{N_x}^{s-\frac{n}{2},2} \ell_{N_y}^{n(\frac{1}{2}-\frac{s_y}{n})+s_y,2}\bg)
L_t^{\frac{\alpha(m+2)}{2}} L_x^{q_\alpha} L_y^2}\\
&\lesssim \|P_R u\|_{\ell_R^{s,2} L_t^{\frac{\alpha(m+2)}{2}} L_x^{q_\alpha} L_y^2}^{\frac{\alpha}{2}}
\lesssim \|P_R u\|_{Z^s}^{\frac{\alpha}{2}}.
\end{aligned}
\end{align}
In the rest it still suffices to estimate $III$ and $IV$, $V$ and $VI$ can then be dealt in a similar way. For $III$, we first unify its index in time and $x$-direction:
\[III\lesssim \||u|^{\frac{\alpha}{2}-k} u^k\|_{B_{t,(\frac{1}{m+2}+\frac{7\sigma}{2}+\frac{10\sigma}{m})^{-1},(\frac{1}{m+2}+\frac{7\sigma}{2}+\frac{10\sigma}{m})^{-1}}
^{(2+\frac{10}{m})\sigma}
B_{x,(\frac{1}{m+2}+\frac{7\sigma}{2}+\frac{10\sigma}{m})^{-1},(\frac{1}{m+2}+\frac{7\sigma}{2}+\frac{10\sigma}{m})^{-1}}
^{(15+\frac{7m}{2})\sigma}
L_y^{\frac{n}{2\sigma}}}.\]
Let $s_{t,1}$ be a number satisfying $\sigma_1\gg s_{t,1}\gg\sigma$ and let $s_{x,1},r_{tx,1}$ satisfy  $\frac{\alpha}{2}(\frac{1}{r_{tx,1}}-s_{t,1})
=(\frac{1}{m+2}+\frac{7\sigma}{2}+\frac{10\sigma}{m})-2\sigma$ and $\frac{\alpha}{2}(\frac{1}{r_{tx,1}}-\frac{s_{x,1}}{m})
=(\frac{1}{m+2}+\frac{7\sigma}{2}+\frac{10\sigma}{m})-\frac{15\sigma}{m}$.
Estimating as in \eqref{5.19} and using Lemma \ref{lem 3.6} it follows
\begin{align*}
\begin{aligned}
III\lesssim \|u\|^{\frac{\alpha}{2}}_{\bg(L_t^{r_{tx,1}} B_{x,r_{tx,1},r_{tx,1}}^{2s_{t,1}+s_{x,1}}
\,\cap\, B_{t,r_{tx,1},r_{tx,1}}^{s_{t,1}+s_{x,1}/2} L_x^{r_{tx,1}} \bg)L_y^{\frac{\alpha n}{4\sigma}}}.
\end{aligned}
\end{align*}
By direct computation one verifies that $(r_{tx,1},r_{tx,1})$ is an $(s-\frac{n}{2}-(2s_{t,1}+s_x)+\frac{4\sigma}{\alpha})$-admissible Strichartz pair. Let $q_1$ be the number such that $(r_{tx,1},q_1)$ is Strichartz admissible. Then estimating as in \eqref{5.20} we obtain
\begin{align*}
\begin{aligned}
III&\lesssim \|P_{N_x}^x P_{N_y}^y u\|^{\frac{\alpha}{2}}_{\ell_{N_x}^{(s-\frac{n}{2}-(2s_{t,1}+s_{x,1})+\frac{4\sigma}{\alpha})+2s_{t,1}+s_{x,1},2}
\ell_{N_y}^{\frac{n}{2}-\frac{4\sigma}{\alpha},2}
L_t^{r_{tx,1}}L_x^{q_1} L_y^2}\\
&\qquad+
\|P_{N_x}^x P_{N_y}^y u\|^{\frac{\alpha}{2}}_{\ell_{N_x}^{s-\frac{n}{2}-(2s_{t,1}+s_{x,1})+\frac{4\sigma}{\alpha},2}
\ell_{N_y}^{\frac{n}{2}-\frac{4\sigma}{\alpha},2}
B_{t,r_{tx,1},2}^{s_{t,1}+s_{x,1}/2}L_x^{q_1} L_y^2}\\
&\lesssim
\|P_N u\|^{\frac{\alpha}{2}}_{\ell_{N}^{s,2}
L_t^{r_{tx,1}}L_x^{q_1} L_y^2}
+
\|P_N u\|^{\frac{\alpha}{2}}_{\ell_{N}^{s-(2s_{t,1}+s_{x,1}),2}
B_{t,r_{tx,1},2}^{s_{t,1}+s_{x,1}/2}L_x^{q_1} L_y^2}\lesssim \|u\|_{Z^s}^{\frac{\alpha}{2}}.
\end{aligned}
\end{align*}
Finally we estimate $IV$. Let $s_{t,2},s_{y,2}$ satisfy $\sigma_1\gg s_{t,2}\gg s_{y,2}\gg \sigma$. Define $r_{t,2},r_{y,2}$ by $\frac{\alpha}{2}(\frac{1}{r_{t,2}}-s_{t,2})
=(\frac{1}{m+2}+\frac{7\sigma}{2})-2\sigma$ and $\frac{\alpha}{2}(\frac{1}{r_{y}^2}-\frac{s_{y,2}}{n})
=(\frac{n}{10\sigma})^{-1}-\frac{15\sigma}{n}$. Estimating as in \eqref{5.19} and using Lemma \ref{lem 3.7} it follows
\begin{align}\label{5.23}
\begin{aligned}
IV&\lesssim \|u\|^{\frac{\alpha}{2}}_{L_t^{r_{t,2}}L_x^{\alpha(\frac{1}{m+2}+\frac{2\sigma}{m})^{-1}/2}
B_{y,r_{y,2},\theta r_{y,2}}^{2s_{t,2}+s_{y,2}}
\,\cap\, B_{t,r_{t,2},(1-\theta)r_{t,2}}^{s_{t,2}+s_{y,2}/2}L_x^{\alpha(\frac{1}{m+2}+\frac{2\sigma}{m})^{-1}/2}
L_y^{r_{y,2}}}\\
&\lesssim \|u\|^{\frac{\alpha}{2}}_{L_t^{r_{t,2}}L_x^{\alpha(\frac{1}{m+2}+\frac{2\sigma}{m})^{-1}/2}
B_{y,\theta r_{y,2},\theta r_{y,2}}^{2s_{t,2}+s_{y,2}+n(\theta^{-1}-1)/r_{y,2}}}\\ &\qquad+\|u\|^{\frac{\alpha}{2}}_{B_{t,(1-\theta)r_{t,2},(1-\theta)r_{t,2}}^{s_{t,2}+s_{y,2}/2+((1-\theta)^{-1}-1)/r_{t,2}}L_x^{\alpha(\frac{1}{m+2}+\frac{2\sigma}{m})^{-1}/2}
L_y^{r_{y,2}}},
\end{aligned}
\end{align}
where $\theta=\frac{s_{y,2}}{2s_{t,2}+s_{y,2}}$. By direct computation one verifies that $(r_{t,2},\alpha(\frac{1}{m+2}+\frac{2\sigma}{m})^{-1}/2)$ is an $(s-\frac{n}{2}-\frac{10\sigma}{\alpha}-2s_{t,2})$-admissible Strichartz pair. Notice also that since $\sigma\ll s_y\ll s_t\ll 1$, the number $\theta r_{y,2}$ satisfies $\theta r_{y,2}=\frac{\alpha n}{(\alpha-{\sigma}/{s_{y,2}})(2s_{t,2}+s_{y,2})}\gg 2$. Let $q^2$ be given such that $(r_{t,2},q^2)$ is an admissible Strichartz pair. Then
\begin{align*}
\begin{aligned}
&\,\|u\|_{L_t^{r_{t,2}}L_x^{\alpha(\frac{1}{m+2}+\frac{2\sigma}{m})^{-1}/2}
B_{y,r_{y,2},\theta r_{y,2}}^{2s_{t,2}+s_{y,2}}}\\
\lesssim&\, \|P_{N_x}^x P_{N_y}^y u\|_{\ell_{N_x}^{s-\frac{n}{2}-\frac{10\sigma}{\alpha}-2s_{t,2},2}
\ell_{N_y}^{(2s_{t,2}+s_{y,2}+\frac{n(\theta^{-1}-1)}{r_{y,2}})+(\frac{n}{2}-\frac{n}{\theta r_{y,2}}),2} L_t^{r_{t,2}} L_x^{q^2} L_y^2}\\
\lesssim&\,
\|P_{N_x}^x P_{N_y}^y u\|_{\ell_{N_x}^{s-\frac{n}{2}-\frac{10\sigma}{\alpha}-2s_{t,2},2}
\ell_{N_y}^{(2s_{t,2}+s_{y,2})-s_{y,2}+\frac{10\sigma}{\alpha},2} L_t^{r_{t,2}} L_x^{q^2} L_y^2}\\
\lesssim&\, \|P_N u\|_{\ell_N^{s,2}L_t^{r_{t,2}} L_x^{q^2} L_y^2}\lesssim \|u\|_{Z^s}.
\end{aligned}
\end{align*}
To estimate the second norm in \eqref{5.23}, we first notice that
$$ \frac{2}{(1-\theta)r_{t,2}}+\frac{m}{(\alpha(\frac{1}{m+2}+\frac{2\sigma}{m})^{-1}/2)}
=\frac{m}{2}-(s-\frac{n}{2}-\frac{10\sigma}{\alpha}-2s_{t,2}-\frac{2\theta}{(1-\theta)r_{t,2}})
=:\frac{m}{2}-\kappa.$$
Since $\sigma\ll s_{y,2}\ll s_{t,2}\ll 1$ and $\alpha>\frac{4}{m}$, we know that $\theta\ll 1$ and $(1-\theta)r_{t,2}>2+\frac{4}{m}>m^*\geq 2$, where the number $m^*$ is defined by \eqref{strichartz condition}. In turn, this implies that $\kappa>0$ and hence $((1-\theta)r_{t,2},\alpha(\frac{1}{m+2}+\frac{2\sigma}{m})^{-1}/2)$ is an $(s-\frac{n}{2}-\frac{10\sigma}{\alpha}-2s_{t,2}-\frac{2\theta}{(1-\theta)r_{t,2}})$-admissible Strichartz pair. Let $q_3$ be the number such that $((1-\theta)r_{t,2},q_3)$ is an admissible Strichartz pair. We then obtain
\begin{align*}
\begin{aligned}
&\,\|u\|_{B_{t,(1-\theta)r_{t,2},(1-\theta)r_{t,2}}^{s_{t,2}+s_{y,2}/2+((1-\theta)^{-1}-1)/r_{t,2}}L_x^{\alpha(\frac{1}{m+2}+\frac{2\sigma}{m})^{-1}/2}
L_y^{r_{y,2}}}\\
\lesssim &\,\|P_{N_x}^x P_{N_y}^y u\|_{
\ell_{N_x}^{s-\frac{n}{2}-\frac{10\sigma}{\alpha}-2s_{t,2}-\frac{2\theta}{(1-\theta)r_{t,2}},2}
\ell_{N_y}^{\frac{n}{2}-\frac{n}{r_{y,2}},2}
B_{t,(1-\theta)r_{t,2},2}^{s_{t,2}+s_{y,2}/2+((1-\theta)^{-1}-1)/r_{t,2}}L_x^{q_3} L_y^2}\\
\lesssim&\,\|P_N u\|_{\ell_{N}^{s-(2s_{t,2}+s_{y,2}+\frac{2\theta}{(1-\theta)r_{t,2}}),2}
B_{t,(1-\theta)r_{t,2},2}^{s_{t,2}+s_{y,2}/2+\frac{\theta}{(1-\theta)r_{t,2}}}L_x^{q_3} L_y^2}
\lesssim \|u\|_{Z^s}.
\end{aligned}
\end{align*}
This completes the desired proof.
\end{proof}

By duality, Lemma \ref{lem5.5} implies immediately the following corollary.

\begin{corollary}\label{lem5.5}
For $\alpha>\frac{4}{m}$ and $m\in\Z$ we have
\begin{align*}
\|v^* |u|^{\alpha-k} u^k\|_{(Z^0)'}\lesssim \|v\|_{Z^0}\|u\|_{Z^s}^\alpha,
\end{align*}
where $v^*\in\{v,\bar v\}$.
\end{corollary}

\section{Nonlinear estimates}\label{sec 5+}
We close the estimation for the nonlinear potential in this section. As we shall see in the following computation, the low-high interaction part will be estimated by using the multilinear estimate derived in Section \ref{sec 5}, while the high-low interaction part is handled by using the fractional chain rules given in Section \ref{sec 3.2}.

We decompose the nonlinear potential $\mathcal{N}(u):=|u|^\alpha u$ according to the dyadic numbers as
$$\mathcal{N}(u)=\sum_{N\in 2^{\N}} \mathcal{N}( P_{\leq N}u)
- \mathcal{N}( P_{\leq \frac{N}{2}}u)=:\sum_{N\in 2^{\N}} F^N.$$
Using the Wirtinger derivative we also know that $F^N$ can be written as
\begin{align}\label{6.1}
\begin{aligned}
F^N&=P_N u \int_0^1 \pt_z\mathcal{N}(P_{\leq \frac{N}{2}}u+\theta P_N u)\,d\theta
+\overline{P_N u}\int_0^1 \pt_{\bar z}\mathcal{N}(P_{\leq \frac{N}{2}}u+\theta P_N u)\,d\theta\\
&=:P_N u\times A^N +\overline{P_N u}\times B^N.
\end{aligned}
\end{align}
\eqref{6.1} and Corollary \ref{lem5.5} imply immediately the following nonlinear estimate for $F_K^N := P_K(F^N)$ in the low-high interaction regime $K\lesssim N$. For a proof, see \cite[Lem. 4.1]{KwakKwon}.

\begin{lemma}[\cite{KwakKwon}]\label{low-high}
For $K,N\in 2^{\N}$ we have
$$\|F_K^N\|_{(Z^{-s})'}\lesssim(K/N)^s\|P_Nu\|_{Z^0}\|u\|^{\alpha}_{Z^s}.$$
\end{lemma}

We will need the following lemma in order to deal with the high-low interaction part $K\gg N$.

\begin{lemma}\label{high-low}
For $K,N\in 2^{\N}$ with $K\gg N$ there exists some $\mu,\nu>0$ such that
$$\|F_K^N\|_{(Z^{-s})'}\lesssim(N/K)^\mu N^{-\nu}\|P_{\leq N}u\|_{Z^{s+\nu}}\|u\|^{\alpha}_{Z^s}.$$
\end{lemma}

\begin{proof}
In the following, the number $q$ is given such that $(\frac{\alpha(m+2)}{2},q)$ is Strichartz admissible. As in \cite{KwakKwon}, we firstly estimate $F_K^n$ by using its product structure. Using H\"older, Bernstein and the scalar or vector-valued Mikhlin multiplier theorem (see e.g. the proof of Lemma \ref{lem3.13}) we obtain
\begin{align*}
\begin{aligned}
&\,\|P_N uP_MA^N \|_{L_t^{\frac{2(m+2)}{m+4}}L_x^{(\frac{m+4}{2(m+2)}+\frac{\sigma_1}{m})^{-1}}L_y^2}
\lesssim \|P_N u\|_{L_{t,x}^{p_0}L_y^2} \|P_M A^N\|_{L_t^{\frac{m+2}{2}}L_x^{(\frac{2}{m+2}+\frac{\sigma_1+\sigma}{m})^{-1}}L_y^\infty}\\
\lesssim&\, N^{-s}M^{-\sigma_1+2\sigma}\|P_N u\|_{Z^s}(\|D_x^{\sigma_1+\sigma}P_M^x P_{\leq M}^y (P_M A^N)\|_{L_t^{\frac{m+2}{2}}L_x^{(\frac{2}{m+2}
+\frac{\sigma_1}{m})^{-1}}L_y^{\frac{n}{\sigma}}}\\
&\qquad+
\|D_y^{\sigma_1+\sigma}P_M^y P_{\leq M}^x (P_M A^N)\|_{L_t^{\frac{m+2}{2}}L_x^{(\frac{2}{m+2}+\frac{\sigma_1}{m})^{-1}}L_y^{\frac{n}{\sigma}}})\\
=:&\,N^{-s}M^{-\sigma_1+2\sigma}\|P_N u\|_{Z^s}(I_a+I_b).
\end{aligned}
\end{align*}
We firstly estimate $I_b$. Assume first $\alpha\in(0,1]$. Then using Lemma \ref{lemallprop} (ii), Corollary \ref{cor 3.2} and the definition of the $Z^s$-space we infer that
\begin{align}\label{6.2}
\begin{aligned}
I_b&\lesssim \|u^\alpha\|_{L_t^{\frac{m+2}{2}}L_x^{(\frac{2}{m+2}+\frac{\sigma_1}{m})^{-1}}
B_{y,\frac{n}{2\sigma},\frac{n}{2\sigma}}^{\sigma_1+2\sigma}}
\lesssim
\|u\|^\alpha_{L_t^{\frac{\alpha(m+2)}{2}}L_x^{\alpha(\frac{2}{m+2}+\frac{\sigma_1}{m})^{-1}}
B_{y,\frac{\alpha n}{2\sigma},2}^{{(\sigma_1+2\sigma)}/\alpha}}\\
&
\lesssim \|u\|^\alpha_{L_t^{\frac{\alpha(m+2)}{2}}B_{x,q,2}^{s-\frac{n}{2}-\frac{\sigma_1}{\alpha}} B_{y,2,2}^{\frac{n}{2}+\frac{\sigma_1}{\alpha}}}\lesssim
\|P_R u\|^{\alpha}_{\ell_R^{s,2}
L_t^{\frac{\alpha(m+2)}{2}}
L_x^{q}L_y^2}\lesssim\|u\|_{Z^s}^\alpha.
\end{aligned}
\end{align}
Consider now the case $\alpha\in(1,\infty)$. We use Lemma \ref{lem3.9} and argue as for \eqref{6.2} to obtain
\begin{align*}
\begin{aligned}
I_b&\lesssim \|u\|^{\alpha-1}_{L_t^{\frac{\alpha(m+2)}{2}}
L_x^{(\frac{\frac{2}{\alpha}}{(m+2)}-\frac{\sigma}{2m(\alpha-1)})^{-1}} L_y^{\frac{2(\alpha-1)n}{\sigma}}}
\|u\|_{L_t^{\frac{\alpha(m+2)}{2}} L_x^{(\frac{\frac{2}{\alpha}}{m+2}+\frac{\sigma_1+\frac{\sigma}{2}}{m})^{-1}}
H_y^{\sigma_1+\sigma,\frac{2n}{\sigma}}}\\
&\lesssim
\|u\|^{\alpha-1}_{L_t^{\frac{\alpha(m+2)}{2}}B_{x,q,2}^{s-\frac{n}{2}+\frac{\sigma}{2(\alpha-1)}} B_{y,2,2}^{\frac{n}{2}-\frac{\sigma}{2(\alpha-1)}}}
\|u\|_{L_t^{\frac{\alpha(m+2)}{2}}B_{x,q,2}^{s-\frac{n}{2}-(\sigma_1+\frac{\sigma}{2})} B_{y,2,2}^{\sigma_1+\frac{\sigma}{2}+\frac{n}{2}}}
\lesssim \|u\|_{Z^s}^\alpha.
\end{aligned}
\end{align*}
Next we estimate $I_a$. In the case $\alpha\in(0,1]$ we are able to exploit fully identical arguments as for \eqref{6.2} to infer that
\begin{align*}
\begin{aligned}
I_a&\lesssim
\|u\|^{\alpha}_{L_t^{\frac{\alpha(m+2)}{2}}B_{x,\alpha(\frac{2}{m+2}+\frac{\sigma_1+\sigma}{m})^{-1},
\alpha(\frac{2}{m+2}+\frac{\sigma_1+\sigma}{m})^{-1}}^{\frac{\sigma_1+2\sigma}{\alpha}}
L_y^{\frac{\alpha n}{\sigma}}}\\
&\lesssim
\|u\|^{\alpha}_{L_t^{\frac{\alpha(m+2)}{2}}
B_{x,q,2}^{s-\frac{n}{2}+\frac{\sigma_1+2\sigma-(\sigma_1+\sigma)}{\alpha}}
B_{y,2,2}^{\frac{n}{2}-\frac{\sigma}{\alpha}}}\lesssim \|u\|_{Z^s}^\alpha.
\end{aligned}
\end{align*}
For $\alpha\in(1,\infty)$, we use Lemma \ref{lem3.13+} to obtain
\begin{align*}
\begin{aligned}
I_a&\lesssim \|u\|^{\alpha-1}_{L_t^{\frac{\alpha(m+2)}{2}}
B_{x,\frac{\alpha(m+2)}{2},\frac{\alpha(m+2)}{2}}^{\frac{\sigma}{2(\alpha-1)}}
L_y^{\frac{2(\alpha-1)n}{\sigma}}}
\|u\|_{L_t^{\frac{\alpha(m+2)}{2}}
B_{x,(\frac{\frac{2}{\alpha}}{m+2}+\frac{2\sigma_1+\sigma}{m})^{-1},
(\frac{\frac{2}{\alpha}}{m+2}+\frac{2\sigma_1+\sigma}{m})^{-1}}^{2\sigma_1+\frac{3\sigma}{2}}
L_y^{\frac{2n}{\sigma}}}\\
&\lesssim
\|u\|^{\alpha-1}_{L_t^{\frac{\alpha(m+2)}{2}}B_{x,q,2}^{s-\frac{n}{2}+\frac{\sigma}{2(\alpha-1)}} B_{y,2,2}^{\frac{n}{2}-\frac{\sigma}{2(\alpha-1)}}}
\|u\|_{L_t^{\frac{\alpha(m+2)}{2}}B_{x,q,2}^{s-\frac{n}{2}+\frac{\sigma}{2}} B_{y,2,2}^{\frac{n}{2}-\frac{\sigma}{2}}}
\lesssim \|u\|_{Z^s}^\alpha.
\end{aligned}
\end{align*}
Summing up we conclude that for $\nu>0$ it holds
\begin{align*}
\|P_N u P_M A^N \|_{L_t^{\frac{2(m+2)}{m+4}}L_x^{(\frac{m+4}{2(m+2)}+\frac{\sigma_1}{m})^{-1}}L_y^2}&\lesssim
N^{-s}M^{-\sigma_1+2\sigma}\|P_N u\|_{Z^s}\|u\|_{Z^s}^\alpha\\
&\lesssim
N^{-s}M^{-\sigma_1+2\sigma}N^{-\nu}\|P_{\leq N} u\|_{Z^{s+\nu}}\|u\|_{Z^s}^\alpha
\end{align*}
which in turn implies
\begin{align}\label{6.3}
\begin{aligned}
&\,\|F_K^N\|_{L_t^{\frac{2(m+2)}{m+4}}L_x^{(\frac{m+4}{2(m+2)}+\frac{\sigma_1}{m})^{-1}}L_y^2}\\
\lesssim&\,\sum_{M\in 2^\N}\|P_K(P_N uP_M A^N )\|_{L_t^{\frac{2(m+2)}{m+4}}L_x^{(\frac{m+4}{2(m+2)}+\frac{\sigma_1}{m})^{-1}}L_y^2}\\
\lesssim&\, \sum_{M\sim K}\|P_K(P_N uP_M A^N )\|_{L_t^{\frac{2(m+2)}{m+4}}L_x^{(\frac{m+4}{2(m+2)}+\frac{\sigma_1}{m})^{-1}}L_y^2}\\
\lesssim&\,N^{-s}K^{-\sigma_1+2\sigma}N^{-\nu}\|P_{\leq N} u\|_{Z^{s+\nu}}\|u\|_{Z^s}^\alpha\\
\lesssim&\,(N/K)^{-s}K^{-s-\sigma_1+2\sigma}N^{-\nu}\|P_{\leq N} u\|_{Z^{s+\nu}}\|u\|_{Z^s}^\alpha.
\end{aligned}
\end{align}
Next, we estimate $F_K^N$ by directly applying Lemma \ref{lem3.9} or Lemma \ref{lem3.13+}. Assume additionally that $s+\nu<1+\alpha$ (which can always be satisfied since $s<1+\alpha$). Direct computation yields
\begin{align*}
\begin{aligned}
\|F_K^N\|_{L_t^{\frac{2(m+2)}{m+4}}L_x^{(\frac{m+4}{2(m+2)}-\frac{2\sigma}{m})^{-1}}L_y^2}&
\lesssim K^{-s-\nu+\sigma}
(\|D_x^{s+\nu-\sigma}F^N\|_{L_t^{\frac{2(m+2)}{m+4}}L_x^{(\frac{m+4}{2(m+2)}-\frac{2\sigma}{m})^{-1}}L_y^2}\\
&\qquad+\|D_y^{s+\nu-\sigma}F^N\|_{L_t^{\frac{2(m+2)}{m+4}}L_x^{(\frac{m+4}{2(m+2)}-\frac{2\sigma}{m})^{-1}}L_y^2})\\
&=:K^{-s-\nu+\sigma}(I_c+I_d).
\end{aligned}
\end{align*}
We use Lemma \ref{lem3.13+} to estimate $I_c$:
\begin{align*}
I_c&\lesssim \|P_{\leq N}u\|^{\alpha}_{L_t^{\frac{\alpha(m+2)}{2}}
B_{x,p_0,p_0
}^{\frac{\sigma}{\alpha}}
L_y^{\frac{\alpha n}{\sigma}}}
\|P_{\leq N} u\|_{L_t^{p_0}
B_{x,p_0,p_0}^{s+\nu}
L_y^{(\frac12-\frac{\sigma}{n})^{-1}}}\\
&\lesssim N^\sigma
\|P_{\leq N}u\|^{\alpha}_{L_t^{\frac{\alpha(m+2)}{2}}
B_{x,q,2}^{s-\frac{n}{2}+\frac{\sigma}{\alpha}}
B_{y,2,2}^{\frac{n}{2}-\frac{\sigma}{\alpha}}}
\|P_{\leq N} u\|_{L_t^{p_0}
B_{x,p_0,2}^{s+\nu}
L_y^{2}}\\
&\lesssim  N^{\nu+\sigma} N^{-\nu} \|P_{\leq N}u\|_{Z^{s+\nu}}\|u\|_{Z^{s}}^\alpha.
\end{align*}
We use Lemma \ref{lem3.9} to estimate $I_d$:
\begin{align*}
I_d&\lesssim \|P_{\leq N}u\|^{\alpha}_{L_t^{\frac{\alpha(m+2)}{2}}
L_x^{\alpha(p_0^{-1}-\frac{\sigma}{m})^{-1}}
L_y^{\frac{\alpha n}{\sigma}}}
\|P_{\leq N} u\|_{L_t^{p_0}
L_x^{(p_0^{-1}-\frac{\sigma}{m})^{-1}}
H_y^{s+\nu-\sigma,(\frac12-\frac{\sigma}{n})^{-1}}}\\
&\lesssim N^\sigma
\|P_{\leq N}u\|^{\alpha}_{L_t^{\frac{\alpha(m+2)}{2}}
B_{x,q,2}^{s-\frac{n}{2}+\frac{\sigma}{\alpha}}
B_{y,2,2}^{\frac{n}{2}-\frac{\sigma}{\alpha}}}
\|P_{\leq N} u\|_{L_t^{p_0}
B_{x,p_0,2}^{\sigma}
B_{y,2,2}^{s+\nu-\sigma}}\\
&\lesssim  N^{\nu+\sigma} N^{-\nu} \|P_{\leq N}u\|_{Z^{s+\nu}}\|u\|_{Z^{s}}^\alpha.
\end{align*}
Summing up implies
\begin{align}\label{6.4}
\|F_K^N\|_{L_t^{\frac{2(m+2)}{m+4}}L_x^{(\frac{m+4}{2(m+2)}-\frac{2\sigma}{m})^{-1}}L_y^2}\lesssim (N/K)^{\nu+\sigma}
K^{-s+2\sigma}\|P_{\leq N}u\|_{Z^{s+\nu}}\|u\|_{Z^{s}}^\alpha.
\end{align}
By interpolating \eqref{6.3} and \eqref{6.4} we know that there exists some $\nu>0$ such that
$$\|F_K^N\|_{L_t^{\frac{2(m+2)}{m+4}}L_x^{\frac{2(m+2)}{m+4}}L_y^2}\lesssim
(N/K)^\mu K^{-s}N^{-\nu}\|P_{\leq N}u\|_{Z^{s+\nu}}\|u\|_{Z^{s}}^\alpha,$$
from which the claim follows by combining the definition of the $Z^{-s}$-space.
\end{proof}

By combining the combinatorial arguments in \cite{KwakKwon}, Lemma \ref{low-high} and Lemma \ref{high-low} imply immediately the following nonlinear estimate. For a proof, see \cite[Lem. 4.4]{KwakKwon}.

\begin{lemma}[\cite{KwakKwon}]\label{nonlinear est 2}
We have $\||u|^\alpha u \|_{(Z^{-s})'}\lesssim \|u\|_{Z^s}^{1+\alpha}$.
\end{lemma}

\section{Proof of Theorem \ref{main thm}}
We are now in a position to give the proof of Theorem \ref{main thm}. Notice the different from the proof given in \cite{KwakKwon}, where certain approximation arguments were exploited, we give in this paper a possibly simpler and shorter proof based on the classical fixed point arguments. The crucial observation here is that the involved metric space is indeed complete w.r.t. the metric induced by the $Y^0$-norm. This in turn will enable us to utilize the stronger estimate Lemma \ref{lem5.5-} other than applying the weaker nonlinear estimate Lemma \ref{nonlinear est 2}.

\begin{lemma}\label{lem7.1}
For positive numbers $C_1,C_2\in(0,\infty)$ the set
\begin{align}
\mathcal{K}:=\{u\in Y^s: \|u\|_{Y^s}\leq 2C_1,\,\|u\|_{Z^s}\leq 2C_2\}
\end{align}
is a complete metric space equipped with the metric $\rho(u,v):=\|u-v\|_{Y^0}$.
\end{lemma}

\begin{proof}
Let $(u_n)_{n\in\N}$ be a Cauchy sequence in $\mathcal{K}$ w.r.t. the metric $\rho$. Since $Y^0$ is a Banach space, $(u_n)_n$ admits a strong limit $u$ in $Y^0$. It is left to show that $u\in \mathcal{K}$. For $\kappa\in\R$ recall the $Y^\kappa$-norm is defined by
\[\|u\|_{Y^\kappa}^2=\sum_{z\in \Z^{d}}\|P_{C_z}(e^{-it\Delta}u)\|_{V^2 H_{x,y}^\kappa}^2\sim
\sum_{z\in \Z^{d}}\la z\ra^{2\kappa}\|P_{C_z}(e^{-it\Delta}u)\|_{V^2 L^2_{x,y}}^2.\]
Since $u_n$ converges to $u$ strongly in $Y^0$, we know that for any $\kappa\in \R$ and $z\in\Z^d$ it holds
\[\la z\ra^{\kappa}\|P_{C_z}(e^{-it\Delta}u_n)\|_{V^2 L^2_{x,y}}\to \la z\ra^{\kappa}\|P_{C_z}(e^{-it\Delta}u)\|_{V^2 L^2_{x,y}}\]
as $n\to\infty$. Choosing $\kappa=s$, we obtain by using Fatou's lemma that
\[\|u\|_{Y^s}\leq\liminf_{n\to\infty}\|u_n\|_{Y^s}\leq 2C_1. \]
By the same reasoning one also infers that $\|u\|_{Z^s}\leq 2C_2$. This completes the desired proof.
\end{proof}

We are now ready to give the final proof for Theorem \ref{main thm}.

\begin{proof}[Proof of Theorem \ref{main thm}]
Let $C$ be some universal positive constant and the space $\mathcal{K}$ be defined as in Lemma \ref{lem7.1} with $C_1=2C$ and $C_2=2C\delta$. Since $\|u_0\|_{Y^s}\leq \delta$ we know that
$$\|e^{it\Delta} u_0\|_{Z^s}\leq C \|e^{it\Delta} u_0\|_{Y^s}\leq C\delta\leq C.$$
We define the mapping $\Phi$ as the Duhamel mapping:
\begin{align*}
\Phi(u):=e^{it\Delta}u_0\mp i\int_0^t e^{i(t-z)\Delta_{x,y}}(|u|^\alpha u)(z)\,dz.
\end{align*}
We aim to prove that $\Phi$ defines a contraction mapping on $\mathcal{K}$. Using Lemma \ref{nonlinear est 2} and Lemma \ref{lem 4.1} we obtain
\begin{align*}
\|\Phi(u)\|_{Y^s}&\leq \|e^{it\Delta}u_0\|_{Y^s}+\||u|^\alpha u\|_{(Z^{-s})'}\leq C+\|u\|_{Z^s}^{1+\alpha}\leq C+(2C)^{1+\alpha}\leq 2C,\\
\|\Phi(u)\|_{Z^s}&\leq \|e^{it\Delta}u_0\|_{Z^s}+\||u|^\alpha u\|_{(Z^{-s})'}\leq C\delta+\|u\|_{Z^s}^{1+\alpha}\leq C\delta+(2C\delta)^{1+\alpha}\leq 2C\delta
\end{align*}
by choosing $\delta\ll 1$. Next, using duality, Lemma \ref{lem5.5-} and the embedding $Y^0\hookrightarrow Z^0$ given by Lemma \ref{lem 4.1} we infer that
\begin{align*}
\|\Phi(u)-\Phi(v)\|_{Y^0}&\leq C\||u|^{\alpha}u-|v|^{\alpha}v\|_{(Z^{0})'}
\leq C\sup_{\|w\|_{Z^0}=1}\|w(|u|^{\alpha}u-|v|^{\alpha}v)\|_{L_{t,x,y}^1}\\
&\leq C\sup_{\|w\|_{Z^0}=1}\|w(u-v)(|u|^{\alpha}+|v|^{\alpha})\|_{L_{t,x,y}^1}\leq C\|u-v\|_{Z^0}(\|u\|_{Z^s}+\|v\|_{Z^s})^\alpha\\
&\leq C\delta^\alpha\|u-v\|_{Z^0}\leq C\delta^\alpha\|u-v\|_{Y^0}.
\end{align*}
The existence of a global solution then follows from Lemma \ref{lem7.1} and standard fixed point arguments. For the scattering result, we may simply consider the result for the case $t\to\infty$, the case $t\to -\infty$ follows in the same manner. Define
\begin{align*}
\phi^+:=u_0-i\int_0^\infty e^{-iz\Delta_{x,y}}(|u|^\alpha u)(z)\,dz.
\end{align*}
Using Lemma \ref{embedding lem}, Lemma \ref{dual lem}, Lemma \ref{lem 4.1}, Lemma \ref{nonlinear est 2} and the dominated convergence theorem we obtain
\begin{align*}
\|u(t)-e^{it\Delta}\phi^+\|_{H^s_{x,y}}&\lesssim\bg\| \int_t^\infty e^{i(t-z)\Delta}(|u|^\alpha u)(z)\,dt\bg\|_{Y^s([t,\infty))}
\lesssim \| |u|^\alpha u\|_{(Y^{-s})'([t,\infty))}\\
&\lesssim \| |u|^\alpha u\|_{(Z^{-s})'([t,\infty))}\lesssim \|u\|_{Z^{s}([t,\infty))}^{1+\alpha}\to 0
\end{align*}
as $t\to\infty$. Finally, that the solution is in the class $C(\R;H_{x,y}^s)$ follows by using similar arguments as those used for the scattering part, we omit the repeating details. This completes the proof.
\end{proof}

\subsubsection*{Acknowledgment}
This research was supported by the NSF grant of China (No. 12301301) and the NSF grant of Guangdong (No. 2024A1515010497). The author is grateful to Zehua Zhao for some stimulating discussions.

%\bibliographystyle{acm}
%\bibliography{NLS}

\end{document}